\nonstopmode
\documentclass[10pt, reqno]{amsart}
\usepackage{graphicx}
\usepackage{latexsym}
\usepackage{fancyhdr}
\usepackage{amsmath, amssymb, amsthm}
\usepackage[all]{xy}
\usepackage{pdflscape}
\usepackage{longtable}
\usepackage{rotating}
\usepackage{verbatim}
\usepackage{hyperref}
\hypersetup{
    linkcolor = blue,
    citecolor = blue,
    urlcolor  = blue,
    colorlinks = true,
}

\usepackage{cleveref}
\usepackage{subfigure}
\usepackage{mathrsfs}
\usepackage{mdwlist}
\usepackage{dsfont}
\usepackage{mathtools}
\usepackage{float}
\usepackage{color}
\usepackage{stmaryrd}
\usepackage{parskip}

\usepackage{tkz-base}
\usepackage{pgfplots}
\pgfplotsset{compat=newest}

\usepackage{tkz-euclide}
\usepackage{etoolbox}

\definecolor{teal}{rgb}{0.0, 0.5, 0.5}

\newcounter{mnotecount}[section]

\newcommand{\rmnote}[1]{}

\allowdisplaybreaks

\DeclareFontFamily{U}{mathb}{\hyphenchar\font45}
\DeclareFontShape{U}{mathb}{m}{n}{
      <5> <6> <7> <8> <9> <10> gen * mathb
      <10.95> mathb10 <12> <14.4> <17.28> <20.74> <24.88> mathb12
      }{}
\DeclareSymbolFont{mathb}{U}{mathb}{m}{n}
\DeclareFontSubstitution{U}{mathb}{m}{n}

\theoremstyle{plain}
\newtheorem*{theorem*}{Theorem}
\newtheorem{theorem}{Theorem}[section]
\newtheorem*{lemma*}{Lemma}
\newtheorem{lemma}[theorem]{Lemma}
\newtheorem*{assumption*}{Assumption}

\newtheorem*{proposition*}{Proposition}
\newtheorem{proposition}[theorem]{Proposition}
\newtheorem*{corollary*}{Corollary}
\newtheorem{corollary}[theorem]{Corollary}
\newtheorem*{claim*}{Claim}

\newtheorem*{conjecture*}{Conjecture}
\newtheorem{question}[theorem]{Question}
\newtheorem*{question*}{Question}

\newtheorem*{result*}{Result}

\theoremstyle{definition}
\newtheorem*{definition*}{Definition}
\newtheorem{definition}[theorem]{Definition}
\newtheorem*{example*}{Example}
\newtheorem{example}[theorem]{Example}

\newtheorem*{algorithm*}{Algorithm}
\newtheorem*{remark*}{Remark}
\newtheorem*{remarks*}{Remarks}
\newtheorem{remark}[theorem]{Remark}

\newtheorem*{convention*}{Convention}

\Crefname{l}{Lemma}{Lemmas}    
\Crefname{p}{Proposition}{Propositions}
\Crefname{t}{Theorem}{Theorems}
\Crefname{c}{Corollary}{Corollaries}
\Crefname{r}{Remark}{Remarks}
\Crefname{d}{Definition}{Definitions}
\Crefname{e}{Example}{Examples}
\Crefname{q}{Question}{Questions}

\numberwithin{equation}{section}

\def\al{\alpha}

\def\ga{\gamma}
\def\de{\delta}
\def\ep{\epsilon}

\def\la{\lambda}

\def\si{\sigma}

\def\vh{\varphi}

\def\ps{\psi}

\def\Ga{\Gamma}

\def\La{\Lambda}

\def\Om{\Omega}

\def\C{\mathbb{C}}

\def\I{\mathbb{I}}

\def\N{\mathbb{N}}

\def\R{\mathbb{R}}
\def\Z{\mathbb{Z}}

\def\cA{\mathcal{A}}

\def\cE{\mathcal{E}}

\def\cH{\mathcal{H}}

\def\cL{\mathcal{L}}

\def\cP{\mathcal{P}}

\def\p{\partial}

\def\aa{\mathsf{a}}
\def\bb{\mathsf{b}}
\def\cc{\mathsf{c}}

\def\II{\mathbf{I}}
\def\sol{\mathcal{S}}
\def\eig{\mathcal{E}}

\def\<{\langle}
\def\>{\rangle}
\renewcommand{\o}{\circ}

\def\ol{\overline}
\def\ul{\underline}

\def\Hyp{\on{Hyp}}

\let\on=\operatorname

\newcommand{\sr}[1]%
{\ifmmode{}^\dagger\else${}^\dagger$\fi\ifvmode
\vbox to 0pt{\vss
 \hbox to 0pt{\hskip\hsize\hskip1em
 \vbox{\hsize3cm\raggedright\pretolerance10000
 \noindent #1\hfill}\hss}\vss}\else
 \vadjust{\vbox to0pt{\vss%
 \hbox to 0pt{\hskip\hsize\hskip1em%
 \vbox{\hsize3cm\raggedright\pretolerance10000%
 \noindent #1\hfill}\hss}\vss}}\fi%
}

\providecommand{\mapsfrom}{\kern.2em%
\setbox0=\hbox{$\leftarrow$\kern-.10em\rule[0.26mm]{0.1mm}{1.3mm}}\box0%
\kern.3em}

\title{Continuity of the solution map for hyperbolic polynomials}

\author[Adam Parusi\'nski and  Armin Rainer]
{Adam Parusi\'nski and Armin Rainer}

\address {Adam Parusi\'nski: Universit\'e C\^ote d'Azur,  CNRS,  LJAD, UMR 7351, 06108 Nice, France}
\email{adam.parusinski@univ-cotedazur.fr}

\address{Armin Rainer: Faculty of Mathematics and Geoinformation,
    Institute for Statistics and Mathematical Methods in Economics, E105-04,
TU Wien, Wiedner Hauptstraße 8, 1040 Vienna, Austria}
\email{armin.rainer@tuwien.ac.at}

\begin{document}

\begin{abstract}
    Hyperbolic polynomials are monic real-rooted polynomials.
    By Bronshtein's theorem, the increasingly ordered roots of a hyperbolic polynomial of degree $d$ with 
    $C^{d-1,1}$ coefficients are locally Lipschitz and the solution map ``coefficients-to-roots'' is bounded. 
    We prove continuity of this solution map from 
    hyperbolic polynomials of degree $d$ with $C^d$ coefficients to their increasingly ordered roots 
    with respect to the $C^d$ structure on the source space and the Sobolev $W^{1,q}$ structure, for all $1 \le q<\infty$,
    on the target space. Continuity fails for $q=\infty$.
    As a consequence, we obtain continuity of the local surface area of the roots 
    as well as local lower semicontinuity of the area of the zero sets of hyperbolic polynomials.
    We also discuss applications for the eigenvalues of Hermitian matrices and singular values. 
\end{abstract}

\thanks{This research was funded 
    in part by the Austrian Science Fund (FWF) DOI 10.55776/P32905 and DOI 10.55776/PAT1381823.
For open access purposes, the authors have applied a CC BY public copyright license to any author-accepted manuscript version arising from this submission.}
\keywords{Hyperbolic polynomials, regularity of the roots, continuity of the solution map, Bronshtein's theorem, Tschirnhausen (trans-)form, continuity of surface area,
eigenvalues of Hermitian matrices, singular values}
\subjclass[2020]{
    26C05,   
    26C10,   
    26A16,   
    30C15,   
    46E35,   
    47A55,   
47H30}   
\date{\today}

\maketitle

\setcounter{tocdepth}{1}
\tableofcontents

\clearpage

\section{Introduction}

Determining the optimal regularity of the roots of polynomials whose coefficients depend smoothly on parameters 
is a much studied problem with a long history.
It has important applications in various fields such as partial differential equations and perturbation theory.

The subject started with Rellich's work \cite{Rellich37} on the analytic perturbation theory of linear operators. 
Bronshtein \cite{Bronshtein80} proved Gevrey well-posedness of the hyperbolic Cauchy problem 
with multiple characteristics 
using his result \cite{Bronshtein79} on the Lipschitz continuity of the roots of hyperbolic polynomials. 
Spagnolo \cite{Spagnolo00},  
motivated by his analysis of certain systems of pseudo-differential equations, 
conjectured that the roots of smooth curves of (not necessarily hyperbolic) polynomials admit absolutely continuous parameterizations. 
This conjecture was proved and the optimal Sobolev regularity of the roots was established in a series 
of papers by Parusi\'nski and Rainer \cite{ParusinskiRainerAC,ParusinskiRainer15,Parusinski:2020aa}, 
after the optimal result for radicals had been obtained by 
Ghisi and Gobbino \cite{GhisiGobbino13}. 
For a more comprehensive account of the history of the problem and its ramifications, e.g., in the perturbation theory 
of linear operators, we refer to the recent survey article \cite{Parusinski:2023ab}.

In this paper, we focus on the class of monic \emph{hyperbolic polynomials} for which the regularity problem has a special flavor; 
the general case of monic complex polynomials is treated in \cite{Parusinski:2024ab}. 
A monic real polynomial of degree $d$ is called \emph{hyperbolic} if all its $d$ roots (counted with multiplicities) are real.  
Hyperbolic polynomials appear naturally as the characteristic polynomials of Hermitian matrices for instance.  

We refer by \emph{Bronshtein's theorem} to the statement that  
any continuous system of the roots of a $C^{d-1,1}$ family of hyperbolic polynomials of degree $d$ is 
actually locally Lipschitz continuous (i.e.\ $C^{0,1}$). In general, this is optimal.
Bronshtein \cite{Bronshtein79} originally proved a version for $C^d$ curves of hyperbolic polynomials of degree $d$ 
(under these assumptions the roots can be represented by differentiable functions).
Bronshtein's rather dense proof is hard to follow.
Wakabayashi \cite{Wakabayashi86} gave a complex analytic proof 
of a more general H\"older version of Bronshtein's theorem,
which had been announced by Ohya and Tarama \cite{OhyaTarama86}; 
it was later proved by Tarama \cite{Tarama06} following Bronshtein's original approach.
Kurdyka and P\u{a}unescu \cite{KurdykaPaunescu08} used resolution of singularities to deduce
local Lipschitz continuity of the roots of hyperbolic polynomials with real analytic coefficients.
A simple proof of Bronshtein's theorem, based on the splitting principle, 
which also established 
explicit uniform bounds for the Lipschitz constants of the roots in terms of the 
$C^{d-1,1}$ norms of the coefficients, was given by Parusi\'nski and Rainer \cite{ParusinskiRainerHyp}. 
We will recall this version in \Cref{thm:Bronshtein}. 

Bronshtein's theorem gives rise to
a bounded \emph{solution map} that takes hyperbolic polynomials of degree $d$ with $C^{d-1,1}$ coefficients to $C^{0,1}$ systems of their roots. 
This will be made precise below.

The purpose of this paper is to investigate the continuity of the solution map
and thus answer a question of Antonio Lerario.

\subsection{Hyperbolic polynomials and the solution map}

A monic polynomial of degree $d$,
\[
    P_{\aa}(Z)= Z^d + \sum_{j=1}^d a_{j}Z^{d-j} \in \R[Z],
\]
is called \emph{hyperbolic} if all its $d$ roots are real.
In the following, we will identify the polynomial $P_{\aa}$ with its 
coefficient vector $\aa=(a_1,\ldots,a_d) \in \R^d$. 
Then the set of all hyperbolic polynomials of degree $d$ is identified with the image of 
the map $\si = (\si_1,\ldots,\si_d) : \R^d \to \R^d$, where 
\[
    \si_j(x_1,\ldots,x_d) = (-1)^j \sum_{i_1<\cdots < i_j} x_{i_1}\cdots x_{i_j}
\]
is the $j$-th elementary symmetric function (up to sign). 
By the Tarski--Seidenberg theorem, $\si(\R^d)$ is 
a closed semialgebraic subset of $\R^d$ which we equip with the trace topology.
We denote this space by $\on{Hyp}(d)$ and call it the \emph{space of hyperbolic polynomials of degree $d$}.

For $\aa \in \on{Hyp}(d)$, let
$\la_1^\uparrow(\aa) \le \cdots \le \la_d^\uparrow(\aa)$ denote the increasingly ordered 
roots of $P_{\aa}$. 
Then 
\begin{equation*} 
    \la^\uparrow = (\la_1^\uparrow,\ldots,\la_d^\uparrow) : \on{Hyp}(d) \to \R^d
\end{equation*}
is a continuous map, see \cite[Lemma 4.1]{AKLM98} or, alternatively, \cite[Lemma 6.4]{Parusinski:2024ab} 
combined with \Cref{lem:unordered}.

Let $U \subseteq \R^m$ be open. 
Let $C^{d-1,1}(U,\on{Hyp}(d))$ denote the space of all $C^{d-1,1}$ maps $\aa : U \to \R^d$ such that 
$\aa(U) \subseteq \on{Hyp}(d)$. 
Thus $\aa \in C^{d-1,1}(U,\on{Hyp}(d))$ amounts to a hyperbolic polynomial $P_{\aa}$ of degree $d$ whose coefficients are 
$C^{d-1,1}$ functions defined on $U$.
We equip $C^{d-1,1}(U,\on{Hyp}(d))$ with the trace topology of the natural Fr\'echet topology on $C^{d-1,1}(U,\R^d)$. 
Note that $C^{d-1,1}(U,\on{Hyp}(d))$ is a closed nonlinear subset of $C^{d-1,1}(U,\R^d)$.
Then Bronshtein's theorem (see \Cref{thm:Bronshtein}) implies that the \emph{solution map}
\begin{equation} \label{eq:solmapLip}
   \sol := (\la^\uparrow)_* : C^{d-1,1}(U,\on{Hyp}(d)) \to C^{0,1}(U,\R^d), \quad \aa \mapsto \la^\uparrow \o \aa,
\end{equation}
is well-defined and bounded (i.e., it maps bounded sets to bounded sets).

\subsection{The main results}

We will see in \Cref{ex:Br-continuity} that the solution map 
$\sol : C^{d-1,1}(U,\on{Hyp}(d)) \to C^{0,1}(U,\R^d)$ is \emph{not} continuous: 
the natural topology on the target $C^{0,1}(U,\R^d)$ is too strong.

However, the solution map $\sol$ becomes continuous if we restrict it to $C^{d}(U,\on{Hyp}(d))$, 
carrying the trace topology of the natural Fr\'echet topology on $C^{d}(U,\R^d)$,
and relax the topology on the target space:
for $1 \le q < \infty$, let $C^{0,1}_q(U,\R^d)$ denote the set $C^{0,1}(U,\R^d)$ 
equipped with the trace topology of the inclusion in Sobolev space 
$C^{0,1}(U,\R^d) \to W^{1,q}_{\on{loc}}(U,\R^d)$. See \Cref{sec:spaces} for precise definitions of the function spaces.

The following theorem, which is our main result, solves Open Problem 3.8 in \cite{Parusinski:2023ab}.

\begin{theorem} \label[t]{thm:solmap}
    Let $U \subseteq \R^m$ be open. 
    The solution map    
    \begin{equation*} 
        \sol : C^{d}(U,\on{Hyp}(d)) \to C^{0,1}_q(U,\R^d), \quad \aa \mapsto \la^\uparrow \o \aa,
    \end{equation*}
    is continuous, for all $1\le q<\infty$.
\end{theorem}

As a corollary, we find that the solution map on $C^{d}(U,\on{Hyp}(d))$ is continuous 
into the H\"older space $C^{0,\al}(U,\R^d)$, carrying its natural topology,
for all $0<\al<1$.

\begin{corollary} \label[c]{cor:solmap}
    Let $U \subseteq \R^m$ be open. 
    The solution map    
    \begin{equation*} 
        \sol : C^{d}(U,\on{Hyp}(d)) \to C^{0,\al}(U,\R^d), \quad \aa \mapsto \la^\uparrow \o \aa,
    \end{equation*}
    is continuous, for all $0< \al<1$, but not for $\al=1$.
\end{corollary}

The essential work for the proof of \Cref{thm:solmap} happens in dimension $m=1$ of the parameter space. 
The passage from one to several parameters is rather easy.
The following is the main technical result of the paper.

\begin{theorem} \label[t]{thm:mainhyp}
    Let $I \subseteq \R$ be an open interval.
    Let $\aa_n \to \aa$ in $C^d (I, \on{Hyp}(d))$,
    i.e.,  
    for each relatively compact open interval $I_1 \Subset I$,
    \begin{equation} \label{eq:mainass} 
        \| \aa- \aa_n\|_{C^{d}(\ol I_1,\R^d)} \to 0 \quad \text{ as } n \to \infty.
    \end{equation}
    Then $\{\sol(\aa_n) : n \ge 1\}$ is a bounded set in $C^{0,1}(I, \R^d)$ (with respect to its natural topology)
    and, for each relatively compact open interval $I_0 \Subset I$ and each $1 \le q < \infty$,
    \begin{equation} \label{eq:mainconcl}
        \|\sol(\aa) -  \sol(\aa_n)\|_{W^{1,q}(I_0,\R^d)}  \to 0 \quad \text{ as } n\to \infty. 
    \end{equation}
\end{theorem}

The proof of \Cref{thm:mainhyp} is based on the dominated convergence theorem. The domination follows 
from Bronshtein's theorem which we recall in \Cref{thm:Bronshtein}. We will show in \Cref{thm:ptwhyp} 
that, for almost every $x \in I$, 
\[
    \sol(\aa_n)'(x) \to \sol(\aa)'(x) \quad \text{ as } n \to \infty.
\]
To this end, we will develop a version of Bronshtein's theorem at a single point, see \Cref{thm:ptwkey}. 

In \Cref{sec:rm}, we prove a refinement of \Cref{thm:mainhyp} in which the assumption that 
$\aa_n \to \aa$ in $C^d$ as $n \to \infty$ can be weakened to convergence in $C^p$, where $p$ is the (uniform) maximal multiplicity 
of the roots of $P_\aa$.

Note that by Egorov's theorem \cite{Egoroff:1911aa} we may conclude that 
$\sol(\aa_n)' \to \sol(\aa)'$ almost uniformly on $I$ as $n \to \infty$, 
i.e., for each $\ep>0$ there exists a measurable subset $E \subseteq I$ with $|E| < \ep$ such that 
$\sol(\aa_n)' \to \sol(\aa)'$ uniformly on $I\setminus E$. 
In general, the convergence is not uniform on the whole interval $I$; see \Cref{ex:Br-continuity}. 

For later reference, we state a simple consequence of \Cref{thm:mainhyp}.
Here $\|x\|_2$ denotes the $2$-norm of $x \in \R^d$ and $\|f\|_{L^q(I_0,\R^d)} := \big\| \|f\|_2 \big\|_{L^q(I_0)}$, 
see \Cref{ssec:notation}.

\begin{corollary} \label[c]{cor:se}
    Let $I \subseteq \R$ be an open interval and $I_0 \Subset I$ a relatively compact open subinterval.
    If $\aa_n \to \aa$ in $C^d (I, \on{Hyp}(d))$ as $n \to \infty$, then 
    \begin{align*}
        \big\| \|\sol(\aa)'\|_2 - \|\sol(\aa_n)'\|_2 \big\|_{L^q(I_0)} \to 0 \quad \text{ as } n \to \infty,  
        \intertext{and}
        \|\sol(\aa_n)'\|_{L^q(I_0,\R^d)} \to \|\sol(\aa)'\|_{L^q(I_0,\R^d)}  \quad \text{ as } n \to \infty,
    \end{align*}
    for all $1 \le q < \infty$.
\end{corollary}

\begin{proof} 
    Let us set $\la := \sol(\aa)$ and $\la_{n} := \sol(\aa_n)$. Then 
    \begin{align*}
        &\big| \|\la'\|_{L^q(I_0,\R^d)} - \|\la_n'\|_{L^q(I_0,\R^d)} \big|
        =
        \big| \big\| \|\la'\|_2 \big\|_{L^q(I_0)} - \big\| \|\la_n'\|_2 \big\|_{L^q(I_0)} \big|
        \\
        & \quad \le 
        \big\| \|\la'\|_2 - \|\la_n'\|_2 \big\|_{L^q(I_0)}
        \le \big\| \|\la' - \la_n'\|_2 \big\|_{L^q(I_0)}
        = \|\la' - \la_n'\|_{L^q(I_0,\R^d)}
    \end{align*}
    so that the assertions follow from \eqref{eq:mainconcl}.
\end{proof}

It would be interesting to have quantitative versions of the continuity results.

\begin{question}
    Are the continuous solution maps $\sol$ in \Cref{thm:solmap} and \Cref{cor:solmap} uniformly continuous 
    and, if yes, is there an effective modulus of continuity?
\end{question}

We have the following quantitative result in the case that all roots are simple. The monic hyperbolic polynomials $P_\aa$
of degree $d$ with $d$ simple roots are in one-to-one correspondence with the points $\aa$ in the interior $\on{Hyp}^\o(d)$ of $\on{Hyp}(d)$.

\begin{theorem} \label[t]{t:simple}
    Let $U \subseteq \R^m$ be open and $k\ge 1$. 
    The solution map    
    \begin{equation*} 
        \sol^\o : C^{k}(U,\on{Hyp}^\o(d)) \to C^{k}(U,\R^d), \quad \aa \mapsto \la^\uparrow \o \aa,
    \end{equation*}
    is locally Lipschitz continuous:
    let $U_0 \Subset U$ and $V_0 \Subset \on{Hyp}^\o(d)$ be relatively compact open convex sets 
    and $B$ a bounded subset of $C^k(\ol U_0, V_0)$.
    Then, for all $\aa_1,\aa_2 \in B$, 
    \[
        \|\sol^\o(\aa_1) - \sol^\o(\aa_2)\|_{C^k(\ol U_0,\R^d)} \le C \, \|\aa_1 - \aa_2\|_{C^k(\ol U_0,\R^d)},
    \]
    where $C=C(d,k,B,V_0)$.
\end{theorem}

We do not know if the continuity results in 
\Cref{thm:solmap}, \Cref{cor:solmap}, \Cref{thm:mainhyp}, and \Cref{cor:se}
still hold for the solution map $\sol$ on $C^{d-1,1}(U,\on{Hyp}(d))$ (instead of $C^{d}(U,\on{Hyp}(d))$).

\begin{question}
    Is the solution map $\sol:  C^{d-1,1}(U,\on{Hyp}(d)) \to C^{0,1}_q(U,\R^d)$ continuous, for $1\le q <\infty$?
    Is the solution map $\sol:  C^{d-1,1}(U,\on{Hyp}(d)) \to C^{0,\al}(U,\R^d)$ continuous, for $0\le \al <1$?
\end{question}

In the proof of \Cref{thm:mainhyp}, we need the convergence of the coefficient vectors in $C^d$ 
only on the accumulation points of the preimage under $\aa$ of the discriminant locus.   
If this preimage is the union of 
an open set and a set of measure zero, then for \eqref{eq:mainconcl}
it is enough that $\aa_n \to \aa$ in $C^{d-1,1}$. 
Thus, for a potential counterexample $\aa$ has to meet the discriminant locus in a Cantor-like set with positive measure.

\begin{remark} \label[r]{r:diff}
    If the coefficients are of class $C^d$, as in the setting of \Cref{thm:mainhyp}, the roots of $P_{\aa_n}$ can be chosen as $C^{1}$ functions
    $\la_{n,1},\ldots,\la_{n,d} : I \to \R$, 
    see \cite{ColombiniOrruPernazza12} or \cite[Theorem 2.4]{ParusinskiRainerHyp}. 
    This choice is not necessarily unique. Also, such a choice imposes an order on the roots that may change when the parameter changes. 
    Therefore, if the parameter space is a circle and not an interval, a consistent choice may not be possible; see e.g.  \Cref{rem:eigenvalues}.   
    Moreover, in general, the roots $\la_{n,1},\ldots,\la_{n,d}$ do not converge to a differentiable 
    system of the roots of $P_\aa$ (even just pointwise) as $n\to \infty$, see \Cref{ex:Br-continuity}.
\end{remark}

\subsection{Applications}

We will give several applications of our continuity results 
by highlighting, in particular, several consequences for stability under perturbations. 

\subsubsection{Relation to the results for general polynomials}

In \Cref{ssec:app1},
we will interpret the results for hyperbolic polynomials as special and stronger versions of the general theorems of 
\cite{Parusinski:2024ab}.
In the general case of a complex polynomial $P_{\aa}$ of degree $d$ the coefficient vector $\aa$ is an 
arbitrary element of $\C^d$. Then 
it is natural to consider the \emph{unordered} $d$-tuple of roots because there is no 
canonical choice of a parameterization of the roots by continuous functions. If the parameter space has 
dimension $\ge 2$, then continuous selections of the roots might not even exist.
In contrast to the hyperbolic case, the general theorems of \cite{Parusinski:2024ab} are only valid 
in $W^{1,q}$, for $1 \le q < d/(d-1)$.

\subsubsection{Continuity of the area of the solution map}

In \Cref{ssec:app2}, we will deduce from \Cref{thm:solmap} that, if $\aa_n \to \aa$ in $C^d(U,\on{Hyp}(d))$
as $n \to \infty$, where $U\subseteq \R^m$ is open,  
then the Jacobian $|J(\sol(\aa_n))|$ of $\sol(\aa_n)$ converges to the Jacobian 
$|J(\sol(\aa))|$ of $\sol(\aa)$ in $L^q_{\on{loc}}$, 
for all $1\le q <\infty$ (see \Cref{cor:area2}). 
Combining this with the area formula, we conclude 
that the surface area of the graph of each single root 
$\sol(\aa_n)_j = \la^\uparrow_j \o \aa_n$, for $1 \le j \le d$,
converges locally to the surface area of the graph of $\sol(\aa)_j$ (see \Cref{cor:area4}). 

As a consequence, we find that the area of the zero sets of $C^d$ families of hyperbolic polynomials of degree $d$ 
locally has a lower semicontinuity property:

\begin{corollary} \label[c]{cor:area5intro}
    Let $U \subseteq \R^m$ be open.
    Let $\aa_n \to \aa$ in $C^d(U,\on{Hyp}(d))$ as $n\to \infty$.
    For any relatively compact open $U_0 \Subset U$, 
    consider the zero sets
    \begin{align} \label{eq:ZZn}
        \begin{split}
        Z &= \{ (x,y) \in U_0 \times \R : P_{\aa(x)}(y)=0\} \quad \text{ and }
        \\
        Z_n &= \{ (x,y) \in U_0 \times \R : P_{\aa_n(x)}(y)=0\}, \quad n\ge 1.
        \end{split}
    \end{align}
    Then 
    \begin{align*}
        \liminf_{n \to \infty}  \cH^m(Z_n) \ge  \cH^m(Z).
    \end{align*}
\end{corollary}

Here $\cH^m$ denotes the $m$-dimensional Hausdorff measure.
\Cref{cor:area5intro} will be restated and proved in \Cref{cor:area5}.

Without hyperbolicity, the area of the real zero set is generally not semicontinuous: 
e.g., for the intersections $Z_t$ of Whitney's umbrella $\{(x,y,z) \in \R^3 : x^2 - y^2 z =0\}$ 
with the planes $\{z=t\}$ and the cylinder $\{x^2+y^2 < 1\}$ we have
\[
    \cH^1(Z_t) = 
    \begin{cases}
        0 & \text{ if } t<0,
        \\
        2 & \text{ if } t=0,
        \\
        4 & \text{ if } t>0.
    \end{cases}
\]

\subsubsection{Approximation by hyperbolic polynomials with simple roots}

In \Cref{ssec:app3}, combining our results with a lemma of Wakabayashi \cite{Wakabayashi86}, we will 
obtain the following approximation result (\Cref{cor:approx}): 
for each hyperbolic polynomial $P_\aa$, where $\aa \in C^d(U,\on{Hyp}(d))$, 
there exists a sequence $(\aa_n)_{n \ge 1} \subseteq C^d(U,\on{Hyp}(d))$ such that 
\begin{itemize}
    \item $\aa_n \to \aa$ in $C^d(U,\on{Hyp}(d))$ as $n \to \infty$;
    \item all roots of $P_{\aa_n(x)}$ are simple for all $x \in U$ and all $n \ge 1$;
    \item $\sol(\aa_n) \in C^d(U,\R^d)$, for all $n \ge 1$, and $\sol(\aa_n) \to \sol(\aa)$ in $C^{0,1}_q(U,\R^d)$, 
        for all $1 \le q< \infty$, as $n \to \infty$;
    \item for each relatively compact open $U_0 \subseteq U$, 
        defining the zero sets $Z$ and $Z_n$ as in \eqref{eq:ZZn},
        the limit $\lim_{n\to \infty} \cH^m(Z_n)$ exists and satisfies
        \[
            \lim_{n\to \infty} \cH^m(Z_n)\ge \cH^m(Z).
        \]
\end{itemize}

\subsubsection{Perturbation theory for Hermitian matrices}

In \Cref{ssec:Hermitian}, 
we will apply our results to the eigenvalues of Hermitian matrices.
Ordering the eigenvalues increasingly,
induces a continuous map 
\[
    \la^\uparrow : \on{Herm}(d) \to \R^d
\]
on the real vector space $\on{Herm}(d)$ of complex Hermitian $d \times d$ matrices. 
By Weyl's perturbation theorem (see \Cref{p:Weyl}), we obtain a bounded map
\[
    \eig := (\la^\uparrow)_* : C^{0,1}(U,\on{Herm}(d)) \to C^{0,1}(U,\R^d), \quad A \mapsto \la^\uparrow \o A.
\]

The continuity results for hyperbolic polynomials imply the following result.

\begin{theorem} \label[t]{t:eigmap}
    Let $U \subseteq \R^m$ be open.  Then the map    
    \begin{equation*} 
        \eig : C^{d}(U,\on{Herm}(d)) \to C^{0,1}_q(U,\R^d), \quad A \mapsto \la^\uparrow \o A,
    \end{equation*}
    is continuous, for all $1\le q<\infty$, and the map    
    \begin{equation*} 
        \eig : C^{d}(U,\on{Herm}(d)) \to C^{0,\al}(U,\R^d), \quad A \mapsto \la^\uparrow \o A,
    \end{equation*}
    is continuous, for all $0< \al<1$.
\end{theorem}

\Cref{t:eigmap} will be proved in \Cref{c:eigmap}. 
The map $\eig$ is not continuous with respect to the $C^{0,1}$ topology on the target space, 
as will be seen in \Cref{ex:A} which is based on \Cref{ex:Br-continuity}.

Given that the map $\eig$ is defined and bounded on $C^{0,1}(U,\on{Herm}(d))$,
it is natural to ask whether in \Cref{t:eigmap} one can replace $C^d$ by $C^1$:

\begin{question} \label[q]{q:AC1}
   Is the map $\eig :  C^{1}(U,\on{Herm}(d)) \to C^{0,1}_q(U,\R^d)$ continuous, for $1 \le q < \infty$?
   Is $\eig :  C^{1}(U,\on{Herm}(d)) \to C^{0,\al}(U,\R^d)$ continuous, for $0 < \al < 1$?
\end{question}

We will prove in \Cref{p:Ad=2} that the answer to \Cref{q:AC1} is affirmative in the case $d=2$. 

\subsubsection{Singular values}

In \Cref{ssec:singval}, we will obtain an analogue of \Cref{t:eigmap} for the 
singular values (ordered by size) of $C^{2d}$ families of general complex $D \times d$ matrices 
with $d \le D$ (see \Cref{c:singval}). As in \Cref{q:AC1}, it is natural to ask whether $C^{2d}$ can actually be replaced by $C^1$.

\subsection{On the optimality of the results}

The following example shows that the solution map $\sol : C^{d}(I,\on{Hyp}(d)) \to C^{0,1}(I,\R^d)$, 
where $I \subseteq \R$ is an open interval,  
is not continuous
with respect to the $C^{0,1}$ topology on the target space.

\begin{example}\label[e]{ex:Br-continuity}
    Let $g(x) := x^2$ and $g_n(x) := x^2 + 1/n^2$, $n \ge 1$.
    Then, for all $k\in \N$ and each bounded open interval $I \subseteq \R$,  
    $\|g-g_n\|_{C^k(\ol I)} = 1/n^2 \to 0$ as $n \to \infty$.
    Let $f$ and $f_n$ be the positive square roots of $g$ and $g_n$, respectively: 
    $f(x) := |x|$ and $f_n(x):= \sqrt{x^2 + 1/n^2}$. 
    Then, for each bounded open interval $I \subseteq \R$ containing $0$, 
    \begin{align*}
        |f- f_n|_{C^{0,1}(\ol I)} &\ge 
        \sup_{0< x \in I} \Big|\frac{(f(x) - f_n(x)) - (f(0)-f_n(0))}{x} \Big| 
        \\
                                  &=  \sup_{0< x \in I} \Big|\frac{x - \sqrt{x^2 + \frac{1}{n^2}} + \frac{1}n}{x}\Big|
                                  \ge \Big|\frac{\frac{1}n- \sqrt{\frac{1}{n^2}+ \frac{1}{n^2}} + \frac{1}n}{\frac{1}n}\Big| 
                                  = 2-\sqrt 2,
    \end{align*}
    for large enough $n$. 
    Observe that
    \[
        f_n'(x) = \frac{x}{\sqrt{x^2 + \frac{1}{n^2}}}
    \]
    tends pointwise to $f'(x) = \on{sgn}(x)$ for all $x \ne 0$ but not uniformly on any neighborhood of $0$:
    \[
        f_n'(\pm \tfrac{1}{n}) = \pm \frac{1}{\sqrt 2}.
    \]
    This also violates the first conclusion of \Cref{cor:se} for $q=\infty$.
\end{example}

Notice that this example also shows that 
not every continuous (thus $C^{0,1}$) system of the roots of $g$ is the limit of a continuous system of the roots of $g_n$:
each continuous system of the roots of $g_n$ tends to $\pm |x|$, none to $\pm x$. See \Cref{r:diff}.

In the example, the hyperbolic polynomial $Z^2 = g(x)$ with double root at $x=0$ is approximated by the 
hyperbolic polynomials $Z^2 = g_n(x)$ with simple roots for all $x$. 
We will see in \Cref{cor:approx} that such an approximation is always possible.

\subsection{Structure of the paper}

We fix notation and recall facts on function spaces in \Cref{sec:spaces} and 
provide the necessary background on hyperbolic polynomials in \Cref{sec:poly}.
In \Cref{sec:Bronshtein}, we recall Bronshtein's theorem in \Cref{thm:Bronshtein} 
and prove a version of it at a single point in \Cref{thm:ptwkey}.
The latter provides bounds for the derivatives of the roots that are crucial for the proof of \Cref{thm:mainhyp}
which is carried out in \Cref{sec:proof1}.
In \Cref{sec:proof2}, we generalize \Cref{thm:mainhyp} to several variables in \Cref{thm:multhyp} 
which allows as to complete the proofs of 
\Cref{thm:solmap} and \Cref{cor:solmap}; also \Cref{t:simple} is proved in \Cref{sec:proof2}. 
\Cref{sec:app} is dedicated to the applications; in particular, it contains the proofs of 
\Cref{cor:area5intro} and \Cref{t:eigmap}.
Finally, \Cref{sec:rm} presents a refinement of \Cref{thm:mainhyp}, namely \Cref{t:rm},
in the case that the maximal multiplicity 
of the roots is smaller than the degree.

\subsection{Notation} \label{ssec:notation}

The $m$-dimensional Lebesgue measure in $\R^m$ is denoted by $\cL^m$.  
If not stated otherwise, `measurable' means `Lebesgue measurable' and `almost everywhere' means 
`almost everywhere with respect to Lebesgue measure'. For measurable $E \subseteq \R^m$, we usually write 
$|E|=\cL^m(E)$.
We will also use the $k$-dimensional Hausdorff measure $\cH^k$.

For $1 \le p \le \infty$, $\|x\|_p$ denotes the $p$-norm of $x \in \R^d$. 
If $f : E \to \R^d$, for measurable $E \subseteq \R^m$, is a measurable map, then we set 
\[
    \|f\|_{L^p(E,\R^d)} := \big\|\|f\|_2 \big\|_{L^p(E)}. 
\]
In the following, a set is called \emph{countable} if it is either finite or has the cardinality of $\N$.

To avoid confusion, 
coefficient vectors of hyperbolic polynomials are 
written in \emph{sans serif} type. For example,
the coefficient vector $\aa_n = (a_{n,1},a_{n,2}, \ldots, a_{n,d})$, indexed by $n \in \N$, 
is notationally distinguished from the scalar $a_n$,
which denotes the $n$-th component of the coefficient vector $\aa$. 

We use the notation $C(d,\ldots)$ to denote a constant that depends only on $d,\ldots$; its value may change 
from line to line.

\section{Function spaces} \label{sec:spaces}

Let us fix notation and recall background on the function spaces used in this paper.

\subsection{H\"older--Lipschitz spaces}

Let $U \subseteq \R^m$ be open and $k \in \N$.
Then $C^k(U)$ is the space of $k$-times continuously differentiable real valued functions, equipped
with its natural Fr\'echet topology.
If $U$ is bounded, then $C^k(\ol U)$ denotes the space of all $f \in C^k(U)$
such that each $\p^\al f$, $0\le |\al|\le k$, 
has a continuous extension to the closure $\ol U$. Endowed with the norm
\[
    \|f\|_{C^k(\ol U)} := \max_{|\al|\le k} \sup_{x \in U} |\p^\al f(x)|
\]
it is a Banach space.
For $0 < \ga \le 1$, we consider the H\"older--Lipschitz seminorm
\[
    |f|_{C^{0,\ga}(\ol U)} := \sup_{x,y \in U, \, x \ne y}\frac{|f(x)-f(y)|}{\|x-y\|_2^\ga}.
\]
For $k \in \N$ and $0 < \ga \le 1$, we have the Banach space  
\[
    C^{k,\ga}(\ol U) := \{f \in C^k(\ol U) : \|f\|_{C^{k,\ga}(\ol U)}  < \infty\},
\]
where 
\[
    \|f\|_{C^{k,\ga}(\ol U)} := \|f\|_{C^k(\ol U)} + \max_{|\al|=k} |\p^\al f|_{C^{0,\ga}(\ol U)}.
\]
We write $C^{k,\ga}(U)$ for the space of $C^k$ functions on $U$ that 
belong to $C^{k,\ga}(\ol V)$ for each relatively compact open $V \Subset U$,
and endow $C^{k,\ga}(U)$ with its natural Fr\'echet topology.

\subsection{Lebesgue spaces}

Let $U \subseteq \R^m$ be open and  $1 \le p \le \infty$.
We denote by $L^p(U)$ the Lebesgue space with respect to the $m$-dimensional Lebesgue measure $\cL^m$, 
and $\| \cdot \|_{L^p(U)}$ is the corresponding $L^p$-norm.
We will also use the space $L^p_{\on{loc}}(U)$ of measurable functions $f : U \to \R$ satisfying
$\|f\|_{L^p(K)} < \infty$ for all compact subsets $K \subseteq U$.
For Lebesgue measurable sets $E \subseteq \R^m$ we also write $|E| = \cL^m(E)$. 
We remark that for continuous functions $f : U \to \R$ we have (and use interchangeably)  
$\|f\|_{L^\infty(U)} = \|f\|_{C^0(\ol U)}$.

\subsection{Sobolev spaces}

For $k \in \N$ and $1 \le q \le \infty$, 
we consider the Sobolev space 
\[
    W^{k,q}(U) := \{f \in L^q(U) : \p^\al f \in L^q(U) \text{ for } |\al|\le k\},
\]
where $\p^\al f$ are distributional derivatives. Endowed with the norm 
\[
    \|f\|_{W^{k,q}(U)} := \sum_{|\al|\le k} \|\p^\al f\|_{L^q(U)}
\]
it is a Banach space. 
We will also use  
\[
    W^{k,q}_{\on{loc}}(U) := \{f \in L^q_{\on{loc}}(U) : \p^\al f \in L^q_{\on{loc}}(U) \text{ for } |\al|\le k\}
\]
and endow this space
with its natural topology.

\subsection{A result on composition}

In the following proposition we use the norm  
\begin{equation*} 
    \|f\|_{C^k(\ol U,\R^\ell)} := \max_{0\le j \le k} \sup_{x \in U} \|d^j f(x)\|_{L_j(\R^m,\R^\ell)} 
\end{equation*}
on the space $C^{k}(\ol U,\R^\ell) := (C^k(\ol U,\R))^\ell$, where $U \subseteq \R^m$ 
and $L_j(\R^m,\R^\ell)$ is the space of $j$-linear maps with $j$ arguments in $\R^m$ and values in $\R^\ell$.

\begin{proposition} \label[p]{prop:lefttr}
    Let $U \subseteq \R^m$ and $V \subseteq \R^\ell$ be open, bounded, and convex.
    Let $\ps \in C^{k+1}(\ol V,\R^p)$. Then 
    \[
        \ps_* : C^k(\ol U,V) \to C^k(\ol U,\R^p), \quad \vh  \mapsto \ps \o \vh, 
    \]
        is well-defined and continuous.
    More precisely, for $\vh_1,\vh_2$ in a bounded subset $B$ of $C^k(\ol U,V)$,
    \[
        \| \ps_*(\vh_1) - \ps_*(\vh_2)\|_{C^k(\ol U,\R^p)} \le C \, \|\ps\|_{C^{k+1}(\ol V,\R^p)} \|\vh_1-\vh_2\|_{C^k(\ol U,\R^\ell)},
    \]
    where $C=C(k,B)$.
\end{proposition}

A short proof of this result can be found in \cite[Appendix A.2]{Parusinski:2024ab}.

\section{Hyperbolic polynomials} \label{sec:poly}

In this section, we recall basic facts on hyperbolic polynomials that will be used below.
The exposition follows \cite{ParusinskiRainerHyp} and \cite{Parusinski:2023ab}. 
For the convenience of the reader and to keep the paper largely self-contained,
we include details where this does not substantially interrupt the flow.

\subsection{Tschirnhausen form}

We say that a monic polynomial 
\[
    P_{\aa}(Z) = Z^d + \sum_{j=1}^d a_j Z^{d-j}
\]
is in \emph{Tschirnhausen form} if $a_1=0$.
Every polynomial $P_{\aa}$ can be put in Tschirnhausen form by the substitution 
\[
    P_{\tilde {\aa}}(Z) = P_{\aa}(Z - \tfrac{a_1}d) = Z^d + \sum_{j=2}^d \tilde a_j Z^{d-j},
\]
which is called the \emph{Tschirnhausen transformation}.
For clarity, we consistently equip the coefficients of polynomials in Tschirnhausen form with a `tilde'.
Note that 
\begin{equation} \label{eq:Tschirn}
    \tilde a_j  = \sum_{i=0}^j C_i a_i a_1^{j-i}, \quad 2 \le j \le d,
\end{equation}
where $a_0 = 1$ and the $C_i$ are universal constants independent of $\aa$.
For a polynomial $P_{\tilde {\aa}}$ in Tschirnhausen form with coefficient vector 
$\tilde {\aa} = (0,\tilde a_2,\ldots,\tilde a_d)$
we have 
\begin{equation} \label{eq:sos}
    -2 \tilde a_2 = \la_1^2 + \cdots + \la_d^2,
\end{equation}
where $\la_1,\ldots,\la_d$ is an enumeration of the roots of $\tilde \aa$.
Consequently, for a hyperbolic polynomial $P_{\tilde \aa}$ in Tschirnhausen form, 
\begin{equation*}
    \tilde a_2 \le 0.
\end{equation*}
Recall that the coefficients (up to their sign) are the elementary symmetric polynomials in the roots, by Vieta's formulas.

\begin{lemma}[{\cite[Lemma 2.4]{Parusinski:2023ab}}] \label[l]{lem:dominant}
    The coefficients of a hyperbolic polynomial $P_{\tilde \aa}$ in Tschirnhausen form satisfy
    \begin{equation*}
        |\tilde a_j|^{1/j} \le \sqrt 2\, |\tilde a_2|^{1/2}, \quad j = 2,\ldots,d.
    \end{equation*}
\end{lemma}

\begin{proof}
    This follows easily from the Newton identities 
    \[
        j \si_j = \sum_{i=1}^j (-1)^{i-1} \si_{j-i} s_i, \quad d \ge j \ge 1,
    \]
    between the Newton polynomials $s_k = \la_1^k + \cdots + \la_d^k$ 
    and the elementary symmetric polynomials $\si_k$,
    observing that $|s_k|^{1/k} \le  |s_2|^{1/2}$, for $2 \le k \le d$, by a well-known relation between the $p$-norms.
\end{proof}

As a consequence, $\tilde \aa = (0,\tilde a_2,\tilde a_3,\ldots,\tilde a_d) =0$ if and only if $\tilde a_2=0$.

\begin{definition}[Spaces of hyperbolic polynomials]
Let $\on{Hyp}_T(d)$ denote the space of monic hyperbolic polynomials of degree $d$ in Tschirnhausen form and $\on{Hyp}_T^0(d)$
the compact subspace of polynomials $P_{\tilde \aa}$ with $\tilde a_2=-1$, i.e.,
\begin{align*}
    \on{Hyp}_T(d) &= \{\tilde \aa \in \on{Hyp}(d) : \tilde a_1=0\},
    \\
    \on{Hyp}_T^0(d) &= \{\tilde \aa \in \on{Hyp}_T(d) : \tilde a_2=-1\}.
\end{align*}
\end{definition}

\subsection{Splitting}

Let us recall a simple consequence of the inverse function theorem.    

\begin{lemma}[E.g.\ {\cite[Lemma 2.5]{Parusinski:2023ab}}] \label[l]{lem:splitting}
    Let $P_{\aa} = P_{\bb} P_{\cc}$, 
    where $P_{\bb}$ and $P_{\cc}$ are monic real polynomials without common (complex) root.
    Then 
    we have $P=P_{\bb(P)}P_{\cc(P)}$ for analytic mappings $P \mapsto \bb(P) \in \R^{\deg P_{\bb}}$ 
    and $P \mapsto \cc(P) \in \R^{\deg P_{\cc}}$, 
    defined for $P$ near $P_{\aa}$ in $\R^{\deg P_{\aa}}$, with the given initial values.
\end{lemma}

\begin{proof}
The product $P_{\aa} = P_{\bb} P_{\cc}$ defines on the coefficients a polynomial map 
$\vh$ such that $\aa = \vh(\bb,\cc)$.
Its Jacobian determinant 
equals the resultant of $P_{\bb}$ and $P_{\cc}$ which is nonzero, by assumption.
    Thus $\vh$ can be inverted locally, by the inverse function theorem.
\end{proof}

Let $P_{\tilde \aa} \in \Hyp_T(d)$ be such that $\tilde \aa \ne 0$, equivalently, $\tilde a_2 \ne 0$.
Then the polynomial
\[
    Q_{\ul \aa}(Z):= |\tilde a_2|^{-d/2} P_{\tilde \aa}(|\tilde a_2|^{1/2} Z) = 
    Z^d - Z^{d-2} + \sum_{j=3}^d |\tilde a_2|^{-j/2} \tilde a_j Z^{d-j}
\]    
belongs to $\Hyp_T^0(d)$.
By \Cref{lem:splitting},
we have a splitting 
\[
    Q_{\ul \aa} = Q_{\ul \bb} Q_{\ul \cc},
\]
on some open neighborhood $\ul U \subseteq \R^d$ of $\ul \aa$
such that $d_{\bb} := \deg Q_{\ul \bb} <d$, $d_{\cc}:=\deg Q_{\ul \cc} <d$, and
\[
\ul b_i = \ps_i(|\tilde a_2|^{-3/2}\tilde a_3,\ldots,|\tilde a_2|^{-d/2}\tilde a_d), \quad i = 1,\dots, \deg Q_{\ul \bb},
\]
where $\ps_i$ are real analytic functions; likewise for $\ul c_i$. 
If $Q_{\ul \aa}$ is hyperbolic, then also $Q_{\ul \bb}$ and $Q_{\ul \cc}$ are hyperbolic.
If $\ul \la_1 \le \cdots \le \ul \la_d$ are the roots of $Q_{\ul \aa}$, then 
we may assume that, on 
$\ul U \cap \Hyp_T^0(d)$,
$\ul \la_1 \le \cdots \le \ul \la_{d_\bb}$ are the roots of $Q_{\ul \bb}$ and 
$\ul \la_{d_\bb+1} \le \cdots \le \ul \la_{d}$ are the roots of $Q_{\ul \cc}$; 
this follows from continuity of the map $\la^\uparrow$ and 
the simple topology of $\on{Hyp}_T^0(d)$ (induced by the embedding in $\R^d$); cf.\ \cite[Theorem 8.1]{Parusinski:2023ab}.

The splitting $Q_{\ul \aa} = Q_{\ul \bb} Q_{\ul \cc}$ induces a splitting
\[
    P_{\tilde \aa} = P_{\bb} P_{\cc}
\]
on an open neighborhood $\tilde U$ of $\tilde \aa$, 
where 
\begin{equation} \label{eq:formulasbhyp}
    b_i = |\tilde a_2|^{i/2} \ps_i(|\tilde a_2|^{-3/2}\tilde a_3,\ldots,|\tilde a_2|^{-d/2}\tilde a_d), \quad i = 1,\dots, \deg P_{\bb}.
\end{equation}
The coefficients $\tilde b_i$ of $P_{\tilde \bb}$, resulting from $P_{\bb}$ by the Tschirnhausen transformation, 
have an analogous representation, i.e., 
\begin{equation} \label{eq:formulasbtihyp}
    \tilde b_i = |\tilde a_2|^{i/2} \tilde \ps_i(|\tilde a_2|^{-3/2}\tilde a_3,\ldots,|\tilde a_2|^{-d/2}\tilde a_d), \quad i = 1,\dots, \deg P_{\bb}.
\end{equation} 
Shrinking $\tilde U$ slightly, we may assume that all partial derivatives of all orders of the real analytic functions $\ps_i$ 
and $\tilde \ps_i$ are bounded on $\tilde U$.

Furthermore, since the roots of $P_{\tilde \aa}$ are given by $\la_j :=|\tilde a_2|^{1/2} \cdot \ul \la_j$, 
for $1\le j \le d$, 
we have that, 
on 
$\tilde U \cap \Hyp_T(d)$,
$\la_1 \le \cdots \le \la_{d_{\bb}}$ are the roots of $P_{\bb}$ and 
$\la_{d_{\bb}+1} \le \cdots \le  \la_{d}$ are the roots of $P_{\cc}$.

\begin{lemma}[{\cite[Lemma 3.13]{Parusinski:2023ab}}] \label[l]{lem:b2a2hyp}
    In this situation, we have $|\tilde b_2| \le 4\, |\tilde a_2|$.  
\end{lemma}

\begin{proof}
    Using \eqref{eq:sos} and $|b_1| \le \sum_{j=1}^{d_\bb}|\la_j| \le \sqrt d_\bb\, \big(\sum_{j=1}^{d_\bb} \la_j^2\big)^{1/2}$, we find
    \begin{align*}
        2 \, |\tilde b_2| &= \sum_{j=1}^{d_\bb} \Big(\la_j + \frac{b_1}{d_\bb}\Big)^2
        = \sum_{j=1}^{d_\bb} \la_j^2  + \frac{2b_1}{d_\bb} \sum_{j=1}^{d_\bb} \la_j + \frac{b_1^2}{d_\bb} \le (1 + 2 + 1) \sum_{j=1}^{d_\bb} \la_j^2 \le 8\, |\tilde a_2|.
    \end{align*}
\end{proof}

\subsection{Universal splitting}

For each $d \ge 2$ fix the following data.    
Choose a finite cover of $\Hyp_T^0(d)$ by open sets $\ul U_1,\ldots,\ul U_s$ 
such that on each $\ul U_i$ we have a splitting 
$Q_{\ul \aa} = Q_{\ul \bb}Q_{\ul \cc}$ and, consequently, a splitting
$P_{\tilde \aa} = P_{\bb} P_{\cc}$ as above
together with analytic functions $\ps_i$ and $\tilde \ps_i$, 
and we fix this splitting. As seen above, we may assume that the roots $\la_1 \le \cdots \le \la_d$ of $P_{\tilde \aa}$ 
are labelled such that $\la_1 \le \cdots \le \la_{d_{\bb}}$ are the roots of $P_{\bb}$ and 
$\la_{d_{\bb}+1} \le \cdots \le  \la_{d}$ are the roots of $P_{\cc}$.

By the Lebesgue covering lemma, there exists $\de>0$ such that each subset of $\Hyp_T^0(d)$ 
of diameter less that $\de$ is contained in some $\ul U_i$. 
Choose $r \in (0, \min\{\de/2,1\})$.
Then 
for each $\ul {\mathsf p} \in \Hyp_T^0(d)$
there exists $1 \le i \le s$ with
\begin{equation} \label{eq:radius}
    B(\ul {\mathsf p}, r) \cap \Hyp_T^0(d) \subseteq \ul U_i\cap \Hyp_T^0(d).
\end{equation}

\begin{definition}[Universal splitting] \label[d]{def:univsplit}
    We refer to this data as a \emph{universal splitting of hyperbolic polynomials of degree $d$ in 
    Tschirnhausen form} and to $r$ as the \emph{radius of the splitting}.
\end{definition}

\section{Bronshtein's theorem and a variant at a single point} \label{sec:Bronshtein}

We recall Bronshtein's theorem in \Cref{thm:Bronshtein}.
We shall need a version at a single point with a suitable bound for the 
derivative of the roots. This version is given in \Cref{thm:ptwkey}.

\subsection{Bronshtein's theorem}

The following result is a version of Bronshtein's theorem \cite{Bronshtein79} 
with uniform bounds
due to \cite{ParusinskiRainerHyp}, see also \cite[Theorem 3.2]{Parusinski:2023ab}.

\begin{theorem} \label[t]{thm:Bronshtein}
    Let $I \subseteq \R$ be an open interval 
    and $\aa \in C^{d-1,1}(I,\on{Hyp}(d))$.     
    Then any continuous root $\la \in C^0(I)$ of $P_{\aa}$ is locally Lipschitz 
    and, for any pair of relatively compact open intervals $I_0 \Subset I_1 \Subset I$,
    \begin{equation} \label{eq:Lipconst}
        |\la|_{C^{0,1}(\ol I_0)}\le C\, \max_{1 \le j \le d} \|a_j\|^{1/j}_{C^{d-1,1}(\ol I_1)},
    \end{equation}
    with $C = C(d) \, \max \{\delta ^{-1}, 1\}$, where $\de := \on{dist}(I_0, \R \setminus I_1)$. 
\end{theorem}

A multiparameter version follows easily; see \cite{ParusinskiRainerHyp} and \cite[Theorem 3.4]{Parusinski:2023ab}.

Note that Wakabayashi~\cite{Wakabayashi86} proved a H\"older version of \Cref{thm:Bronshtein} 
(without uniform bounds of the type \eqref{eq:Lipconst}), 
see also Tarama \cite{Tarama06} for a different proof.

\subsection{Reclusive points}

The local version of Bronshtein's theorem, \Cref{thm:ptwkey}, holds at all points of $I$ except for a countable subset of points, 
which we call reclusive points. A point $x\in I$ is reclusive if either all the roots of $P_{\tilde \aa(x)}$ are zero and $x$ is 
isolated for this property, or it satisfies a similar condition for one of the local factors of $P_{\tilde \aa(x)}$, 
see \Cref{def:recpoints} for a precise formulation. 

\begin{definition}[Zero sets]
    Let $\tilde \aa : I \to \Hyp_T(d) \subseteq \R^d$. We consider the zero set
    \[
        Z_{\tilde \aa} := \{x \in I : \tilde \aa(x) = 0\}= \{x\in I : \text{ all roots of $P_{ \aa(x)}$ coincide}\} 
    \]
    which coincides with $Z_{\tilde a_2} = \{x \in I : \tilde a_2(x) = 0\}$, by \Cref{lem:dominant}. 
    (For notational simplicity, we will generally use $Z_{\tilde \aa}$.)
    We write $\on{acc}(Z_{\tilde \aa})$ and $\on{iso}(Z_{\tilde \aa}) := Z_{\tilde \aa} \setminus \on{acc}(Z_{\tilde \aa})$ 
    for the sets of accumulation points and isolated points of $Z_{\tilde \aa}$, respectively.
\end{definition}

\begin{lemma} \label[l]{l:zero}
    Let $\tilde \aa : I \to \Hyp_T(d) \subseteq \R^d$.
    Then 
    \[
        Z_{\tilde \aa} = \{x \in I : \text{ all roots of } P_{\tilde \aa(x)} \text{ vanish}\}.
    \]
\end{lemma}

\begin{proof}
    Since $Z_{\tilde \aa} = Z_{\tilde a_2}$, this is immediate from \eqref{eq:sos}. 
\end{proof}

Let $I \subseteq \R$ be an open interval
and $\tilde \aa \in C^{d-1,1}(I,\on{Hyp}_T(d))$;
recall that this means $\tilde \aa \in C^{d-1,1}(I,\R^d)$ and $\tilde \aa(I) \subseteq \on{Hyp}_T(d)$.
Let $x_0 \in I$ be such that $\tilde a_2(x_0) \ne 0$.
Then not all roots of $P_{\tilde \aa(x_0)}$ coincide and hence $P_{\tilde \aa}$ splits in a neighborhood of $x_0$. We may assume that it is a \emph{full splitting}, i.e.,
if $\{\la_1,\ldots,\la_k\}$ are the distinct roots of $P_{\tilde \aa(x_0)}$ with multiplicities $\{m_1,\ldots,m_k\}$ 
then 
\begin{equation} \label{eq:fullsplit}
    P_{\tilde \aa} = P_{\bb_1} P_{\bb_2} \cdots P_{\bb_k} \quad \text{ in a neighborhood of } x_0,
\end{equation}
where $\deg P_{\bb_j} = m_j$ and $P_{\bb_j(x_0)}(Z) = (Z-\la_j)^{m_j}$, for all $1 \le j \le k$.
Note that the full splitting is unique up to the order of the factors.
Since the Tschirnhausen transformation $\bb_j \leadsto \tilde \bb_j$ effects a shift of the roots by $b_{j,1}/m_j = -\la_j$, 
we have $x_0 \in Z_{\tilde \bb_j}$, for all $1 \le j \le k$.

\begin{definition}[Reclusive points] \label[d]{def:recpoints}
Let $\tilde \aa \in C^{d-1,1}(I,\on{Hyp}_T(d))$.  
We say that $x_0 \in I$ is \emph{reclusive for $\tilde \aa$} if 
\begin{itemize}
    \item $x_0 \in \on{iso}(Z_{\tilde \aa})$, 
    \item or $x_0 \not \in Z_{\tilde \aa}$ and 
        $x_0 \in \on{iso}(Z_{\tilde \bb_{j}})$ for some $j \in \{1\,\ldots,k\}$, 
        where we refer to the full splitting \eqref{eq:fullsplit}.
\end{itemize}
\end{definition}

Note that, by \Cref{l:zero}, $x_0$ is an isolated point of $Z_{\tilde \bb_{j}}$ if and only if $x_0$ is an isolated point of 
\[
    E_{\bb_j}:=  \{x : \text{ all roots of } P_{\bb_j(x)} \text{ coincide} \}.
\]

\begin{lemma} \label[l]{lem:splitrec}
    Let $I \subseteq \R$ be an open interval
    and $\tilde \aa \in C^{d-1,1}(I,\on{Hyp}_T(d))$.
    Let $x_0 \in I$ be such that $\tilde a_2(x_0) \ne 0$ and assume that $x_0$ is not reclusive for $\tilde \aa$. 
    If $P_{\tilde \aa} = P_{\bb}P_{\cc}$ is any splitting near $x_0$, 
    then $x_0$ is not reclusive for $\tilde \bb$ and $\tilde \cc$ (which result from $\bb$ and $\cc$ by 
     the Tschirnhausen transformation).
\end{lemma}

\begin{proof}
    After possibly reordering the factors in \eqref{eq:fullsplit}, 
    we may assume that, in a neighborhood of $x_0$,
    \[
        P_{\bb} = P_{\bb_1}\cdots P_{\bb_j} \quad \text{ and } \quad P_{\cc} = P_{\bb_{j+1}}\cdots P_{\bb_k}.
    \]
    The Tschirnhausen transformation $\bb \leadsto \tilde \bb$ effects a shift on all roots of 
    $P_{\bb}$ by $b_1/\deg P_{\bb}$ and 
    retains the splitting,
    \[
        P_{\tilde \bb} = P_{\hat \bb_1}\cdots P_{\hat \bb_j}.
    \]
    It follows that $E_{\hat \bb_i} = E_{\bb_i}$ for all $1\le i \le j$.
    Suppose for contradiction that $x_0$ is reclusive for $\tilde \bb$.
    If $x_0$ is an isolated point of $Z_{\tilde \bb}$, then $j=1$, by \Cref{l:zero}, and hence $x_0$ is reclusive for $\tilde \aa$.
    If $x_0 \not\in Z_{\tilde \bb}$ and there is $i \in \{1,\ldots,j\}$ such that $x_0$ is an isolated point 
    of $E_{\hat \bb_i} = E_{\bb_i}$, 
    then again $x_0$ is reclusive for $\tilde \aa$. Since we assumed that $x_0$ is not reclusive for $\tilde \aa$, 
we conclude that $x_0$ is not reclusive for $\tilde \bb$.

The proof that $x_0$ is not reclusive for $\tilde \cc$ is analogous.
\end{proof}

\begin{lemma} \label[l]{lem:countable}
    Let $I \subseteq \R$ be an open interval
    and $\tilde \aa \in C^{d-1,1}(I,\on{Hyp}_T(d))$.
    The set of all $x_0 \in I$ that are reclusive for $\tilde \aa$ is countable.
\end{lemma}

\begin{proof}
    Let $\la := \sol(\tilde \aa)$. Then $\la$ is a curve in $\{y \in \R^d : y_1 \le y_2 \le \cdots \le y_d\}$.
    For $1 \le i < d$, let $\ell_{i}(y) := y_{i+1}-y_i$. 
    If $x_0 \in I$ is reclusive for $\tilde \aa$, then there exist $1 \le i_1<\cdots <i_k < d$ such that 
    $x_0$ is an isolated point of 
    \[
        \{x \in I : \ell_{i_j}(\la(x)) =0 \text{ for all } 1\le j\le k\}.
    \]
    The set of isolated points of the latter set is countable. The statement follows.  
\end{proof}

\subsection{A version of Bronshtein's theorem at a single point}

For $x_0 \in \R$ and $r>0$, let $I(x_0,r)$ denote the open interval centered at $x_0$ with radius $r$, 
\[
    I(x_0,r) := \{x \in \R : |x-x_0|<r\}.
\]
Its closure is denoted by $\ol I(x_0,r)$.

\begin{theorem} \label[t]{thm:ptwkey}
    Let $x_0 \in \R$ and $\de>0$.
    Let $\tilde \aa \in C^{d-1,1}(\ol I(x_0,\de),\on{Hyp}_T(d))$.
    Assume that $x_0$ is not reclusive for $\tilde \aa$.
    Let $\la \in C^0(I(x_0,\de))$ be a continuous root of $P_{\tilde \aa}$ and assume that $\la'(x_0)$ exists.
    Then
    \[
        |\la'(x_0)|\le C(d)\, A(\de),
    \]
    where 
    \begin{align}
        A(\de)&:= 6 \max\{A_1(\de),A_2(\de)\}, \label{eq:defAde} 
        \\
        A_1(\de) &:= \max \big\{ \de^{-1}|\tilde a_{2}(x_0)|^{1/2}, |\tilde a_{2}'|_{C^{0,1}(\ol I(x_0,\de))}^{1/2}\big\}, \nonumber
        \\
        A_2(\de) &:= \max_{2 \le j \le d} \big\{ |\tilde a_{j}^{(d-1)}|_{C^{0,1}(\ol I(x_0,\de))} \cdot \|\tilde a_{2}\|_{L^\infty(I(x_0,\de))}^{(d-j)/2}  \big\}^{1/d}. \nonumber
    \end{align}
    Here $\tilde a_2'$ denotes the first derivative of $\tilde a_2$ and 
    $\tilde a_{j}^{(d-1)}$ the derivative of order $d-1$ of $\tilde a_j$.
\end{theorem}

The proof follows the general strategy of the proof of \Cref{thm:Bronshtein} in \cite{ParusinskiRainerHyp} and \cite{Parusinski:2023ab},
but some modifications are required. 
Before we prove \Cref{thm:ptwkey} let us recall two important tools.

\subsection{Local Glaeser inequality}

Glaeser's inequality \cite{Glaeser63R} gives Bronshtein's theorem in the simplest nontrivial case:
for nonnegative $C^1$ functions $f$ on $\R$ with $f'' \in L^\infty(\R)$ we have 
\[
    f'(x)^2 \le 2\, f(x) \|f''\|_{L^\infty(\R)}, \quad x \in \R.
\]
We need a local version.

\begin{lemma}[{\cite[Lemma 3.14]{Parusinski:2023ab}}] \label[l]{lem:Glaeser}
    Let $I\subseteq \R$ be an open bounded interval.
    Let $f \in C^{1,1}(\ol I)$ satisfy $f \ge 0$ or $f \le 0$ on $I$.
    Assume that $x_0 \in I$ satisfies $f(x_0) \ne 0$
    and let $M>0$ be such that $I_0 := I(x_0,M^{-1} |f(x_0)|^{1/2}) \subseteq I$.    
    Then 
    \[
        |f'(x_0)| \le (M +M^{-1} |f'|_{C^{0,1}(\ol I_0)}) |f(x_0)|^{1/2}.
    \]
    Therefore, if additionally $|f'|_{C^{0,1}(\ol I_0)}\le M^2$, then  
    \[
        |f'(x_0)| \le 2M\, |f(x_0)|^{1/2}.
    \]
\end{lemma}

It should be added that, for a function 
$f \in C^{1,1}(\ol I)$ satisfying $f \ge 0$ or $f \le 0$ on $I$,
we have that $f(x_0)=0$ implies $f'(x_0)=0$, so that the conclusion of the lemma also holds 
trivially at zeros $x_0$ of $f$.

\begin{proof}
    Suppose that $f\ge 0$; otherwise consider $-f$.
    Thus $f(x_0)>0$ and 
    \[
        0 \le f(x_0 + h) = f(x_0) + f'(x_0)h + \int_0^1 f'(x_0+hs) -f'(x_0) \, ds\cdot h. 
    \]
    Setting $h:= \pm M^{-1}|f(x_0)|^{1/2}$, implies the lemma.
\end{proof}

\subsection{Interpolation}

Let us recall an interpolation inequality for intermediate derivatives.

\begin{lemma}[{\cite[Lemma 3.16]{Parusinski:2023ab}}] \label[l]{lem:interpol}
    Let $f \in C^{m,1}(\ol I)$, where $I \subseteq \R$ is a bounded open interval.
    Then, for $1 \le j \le m$, 
    \[
        |f^{(j)}(x)| \le C(m)\, |I|^{-j} \big( \|f\|_{L^\infty(I)} + |f^{(m)}|_{C^{0,1}(\ol I)} |I|^{m+1}\big), \quad x \in I.
    \]
\end{lemma}

\begin{proof}
    Fix $x \in I$. Then $[x,x+|I|/2)$ or $(x-|I|/2,x]$ is contained in $I$.
    Let $y$ be a point in the respective interval. 
    By Taylor's formula,
    \begin{align*}
        \MoveEqLeft \Big| \sum_{j=0}^{m} \frac{f^{(j)}(x)}{j!} (y-x)^j\Big| 
        \\
        &= \Big| f(y) - (y-x)^{m} \int_0^1 \frac{(1-t)^{m-1}}{(m-1)!} (f^{(m)}(x+t(y-x)) - f^{(m)}(x)) \, dt\Big|
        \\
        &\le \|f\|_{L^\infty(I)} + |I|^{m+1} |f^{(m)}|_{C^{0,1}(\ol I)}. 
    \end{align*}
    This implies the lemma in view of the following fact:
    if a polynomial $T(x) = a_0 + a_1 x + \cdots + a_m x^m \in \C[x]$ satisfies $|T(x)| \le A$ for $x \in [0,B] \subseteq \R$,
    then 
    \[
        |a_j| \le C(m)\, A B^{-j}, \quad 0 \le j \le m.
    \]
    Indeed, for $A=B=1$ this follows easily, by comparison with the  
    interpolation polynomial for the equidistant points
    $0 = x_0 < x_1 < \cdots < x_m=1$. In the general case, consider $A^{-1}T(Bx)$.
\end{proof}

\subsection{Proof of \Cref{thm:ptwkey}}

The rest of the section is devoted to the proof of \Cref{thm:ptwkey}.
We may assume that $d \ge 2$, since \Cref{thm:ptwkey} is trivially true for $d=1$.
The following definition will prove convenient for the inductive proof based on the splitting.

We will work on intervals centered at $x_0$ whose radius depends on the size of $\tilde a_2(x_0)$. 
More precisely, assuming that $\tilde a_2(x_0) \ne 0$ we set, for any constant $A>0$,
\begin{equation} \label{eq:II}
        \II(x_0,A) := I(x_0,A^{-1} |\tilde a_2(x_0)|^{1/2})
\end{equation}
and denote by $\ol \II(x_0,A)$ the closure of $\II(x_0,A)$.

\begin{definition}[$C^{d-1,1}$-admissible data] \label[d]{def:ptadm}
    Let $x_0 \in \R$ and $\de>0$.
    Let $\tilde \aa \in C^{d-1,1}(\ol I(x_0,\de),\on{Hyp}_T(d))$ be  
    such that $\tilde a_2(x_0) \ne 0$.
    Let $A>0$ be a constant.
    We say that 
$(\tilde \aa,x_0,\de,A)$ is \emph{$C^{d-1,1}$-admissible} if
    the following holds:
    \begin{enumerate}
        \item  $\II(x_0,A)  \subseteq I(x_0,\de)$. \label{eq:A1pt}
        \item For all $x \in \II(x_0,A)$,
            \begin{equation}\label{eq:A2pt}
          \frac{1}{2} \le \frac{\tilde a_2(x)}{\tilde a_2(x_0)} \le 2.      
            \end{equation}
        \item For all $2\le j \le d$,
            \begin{equation}\label{eq:A4pt}
                |\tilde a_j^{(d-1)}|_{C^{0,1}(\ol \II(x_0,A))} \le A^d \, |\tilde a_2(x_0)|^{(j-d)/2}.
            \end{equation}
        \item For all $2\le j \le d$, $1 \le k \le d-1$, and $x \in \II(x_0,A)$,
            \begin{equation}\label{eq:A3pt}
          |\tilde a_j^{(k)}(x)| \le A^k \, |\tilde a_2(x_0)|^{(j-k)/2}.      
            \end{equation}
    \end{enumerate}
\end{definition}

\begin{lemma} \label[l]{lem:enough}
    Let $x_0 \in \R$ and $A,\de>0$.
    Let $\tilde \aa \in C^{d-1,1}(\ol I(x_0,\de),\on{Hyp}_T(d))$ be  
    such that $\tilde a_2(x_0) \ne 0$.
    Assume that
    \begin{itemize}
        \item $\II(x_0,A)  \subseteq I(x_0,\de)$,
        \item \eqref{eq:A4pt} holds, and
        \item \eqref{eq:A3pt} holds for $2 \le j \le d$ and $k \ge j$.
    \end{itemize}
    Then there is a constant $C(d)\ge 1$ such that $(\tilde \aa,x_0,\de,C(d)A)$ is $C^{d-1,1}$-admissible.
\end{lemma}

\begin{proof}
    We first observe that we have $|\tilde a_2'|_{C^{0,1}(\ol \II(x_0,A))} \le  A^2$. 
    Indeed, if $d=2$ this is immediate from \eqref{eq:A4pt}
    and if $d\ge 3$ then it follows from \eqref{eq:A3pt} with $j=k=2$.
    By Lemma \ref{lem:Glaeser}, we conclude that
    \[
        |\tilde a_2'(x_0)| \le 2  A \, |\tilde a_2(x_0)|^{1/2}.
    \]
    Thus, for $x \in \II(x_0,6A)$, 
    \begin{align*}
       |\tilde a_2(x) - \tilde a_2(x_0)| &\le |\tilde a_2'(x_0)| |x-x_0| + |\tilde a_2'|_{C^{0,1}(\ol \II(x_0,A))}\, |x-x_0|^2
      \\
                                       &\le  \frac{1}{3}\, |\tilde a_2(x_0)| + \frac{1}{36} \, |\tilde a_2(x_0)| < \frac{1}{2}\, |\tilde a_2(x_0)|,
    \end{align*}
    implying \eqref{eq:A2pt} on $\II(x_0,6A)$.

    Finally, we check that \eqref{eq:A3pt} also holds for $k<j$, for $A$ replaced by $C(d)A$, where $C(d) \ge 1$ is a suitable constant (note that $\II(x_0,C(d)A) \subseteq \II(x_0,A)$).
    By \Cref{lem:interpol}, for $1 \le k \le j-1$ and $x \in \II(x_0,6A)$,
    \begin{align*}
        |\tilde a_j^{(k)}(x)| &\le C(d)\, |\II(x_0,6A)|^{-k} \big(\|\tilde a_j\|_{L^\infty(\II(x_0,6A))} + |\tilde a_j^{(j-1)}|_{C^{0,1}(\ol \II(x_0,6A))}\, |\II(x_0,6A)|^{j}\big)
        \\
                              &\le C(d)\, (3A)^k |\tilde a_2(x_0)|^{-k/2}  \big(2^j\,|\tilde a_2(x_0)|^{j/2} + A^j \cdot (3A)^{-j} |\tilde a_2(x_0)|^{j/2}\big)
                              \\
                              &\le \tilde C(d)\, A^k\, |\tilde a_2(x_0)|^{(j-k)/2},
    \end{align*}
    since
    $|\tilde a_j(x)| \le (\sqrt 2\,|\tilde a_2(x)|^{1/2})^j \le 
    2^j\,|\tilde a_2(x_0)|^{j/2}$, by \Cref{lem:dominant} and \eqref{eq:A2pt} on $\II(x_0,6A)$,
    and
    $|\tilde a_j^{(j-1)}|_{C^{0,1}(\ol \II(x_0,6A))} \le A^j$, by \eqref{eq:A4pt} or \eqref{eq:A3pt} for $k=j$.
\end{proof}

\begin{lemma} \label[l]{lem:ptwass>adm}
    Let $x_0 \in \R$ and $\de>0$.
    Let $\tilde \aa \in C^{d-1,1}(\ol I(x_0,\de),\on{Hyp}_T(d))$ be  
    such that $\tilde a_2(x_0) \ne 0$.
    Let $A(\de)$ be defined by \eqref{eq:defAde}. 
    Then $(\tilde \aa,x_0,\de,C(d)A(\de))$ is $C^{d-1,1}$-admissible for some constant $C(d) \ge 1$.
\end{lemma}

\begin{proof}
    (1) By \eqref{eq:defAde}, $A(\de) \ge A_1(\de) \ge  \de^{-1}|\tilde a_{2}(x_0)|^{1/2}$ and thus 
    \begin{equation} \label{eq:IIincl}
        \II(x_0,A(\de)) \subseteq \II(x_0,A_1(\de)) \subseteq I(x_0,\de).
    \end{equation}

    (2)
    By \Cref{lem:Glaeser} and the definition of $A_1(\de)$, 
    \[
        |\tilde a_2'(x_0)| \le 2 A_1(\de) \, |\tilde a_2(x_0)|^{1/2}. 
    \]
    Then, for $x \in \II(x_0, 6A_1(\de))$, we find (as in the proof of \Cref{lem:enough})
    \begin{align*}
        |\tilde a_2(x) -\tilde a_2(x_0)| 
                                       \le \frac{1}{2}|\tilde a_2(x_0)|,
    \end{align*}
    using \eqref{eq:IIincl} and the definition of $A_1(\de)$,
    which implies \eqref{eq:A2pt} with $A=A(\de)$.
    
    (3) By the definition of $A_2(\de)$, \eqref{eq:A4pt} with $A=A(\de)$ is clear.

    (4)
    By \Cref{lem:dominant} and \eqref{eq:A2pt}, we have 
     $|\tilde a_j(x)| \le 
    2^j\,|\tilde a_2(x_0)|^{j/2}$, for $x \in \II(x_0,A(\de))$.
    In conjunction with \eqref{eq:A4pt}, it implies \eqref{eq:A3pt} with $A=C(d)A(\de)$, by 
    \Cref{lem:interpol}. We clearly may assume that $C(d)\ge 1$ so that $\II(x_0,C(d)A(\de)) \subseteq \II(x_0,A(\de))$.
\end{proof}

\begin{lemma} \label[l]{lem:Aconsequence}
    Let $(\tilde \aa,x_0,\de,A)$ be $C^{d-1,1}$-admissible. 
    Then the functions $\ul a_j := |\tilde a_2|^{-j/2} \tilde a_j$, $2 \le j \le d$, are well-defined and of class $C^{d-1,1}$ on $\II(x_0,A)$ and they satisfy
    \begin{align}
          &|\ul a_j^{(d-1)}|_{C^{0,1}(\ol \II(x_0,A))} \le C(d)\, A^d \, |\tilde a_2(x_0)|^{-d/2}, \quad 2 \le j \le d,    \label{eq:A6pt}
          \\
          &|\ul a_j^{(k)}(x)| \le C(d)\,  A^k \, |\tilde a_2(x_0)|^{-k/2},\quad 2\le j \le d,\, 1 \le k \le d-1,\,   x \in \II(x_0,A). \label{eq:A5pt}
    \end{align}
\end{lemma}

\begin{proof}
    By \eqref{eq:A2pt} and \eqref{eq:sos}, $|\tilde a_2| >0$ on $\II(x_0,A)$. Thus
    the functions $\ul a_j$ are well-defined and of class $C^{d-1,1}$ on $\II(x_0,A)$.

    Let $1 \le s \le d-1$ and $r \in \R$. Then 
     Fa\`a di Bruno's formula implies 
  \begin{equation} \label{FaadiBruno}
    \p^{s} \big(\tilde a_2^{r} \big) = 
    \sum_{\ell \ge 1}^{s} \sum_{\ga \in \Ga(\ell,s)} c_{\ga,\ell,r}  
    \, \tilde a_2^{r-\ell} \tilde a_2^{(\ga_1)} \cdots \tilde a_2^{(\ga_\ell)}
  \end{equation}
  where $\Ga(\ell,s) = \{\ga \in \N_{>0}^\ell : |\ga| = s\}$ and
  \[
    c_{\ga,\ell,r}=  \frac{s!}{\ell! \ga!} r(r-1) \cdots (r-\ell+1).
  \]
  Thus, by \eqref{eq:A2pt} and \eqref{eq:A3pt}, for $x \in \II(x_0,A)$,
  \begin{align} \label{eq:derpower}
    |\p^{s} \big(\tilde a_2^{r} \big)(x)| 
    &\le  
    \sum_{\ell \ge 1}^{s} \sum_{\ga \in \Ga(\ell,s)} |c_{\ga,\ell,r}|  
    \, |\tilde a_2^{r-\ell}(x)| |\tilde a_2^{(\ga_1)}(x)|
    \cdots |\tilde a_2^{(\ga_\ell)}(x)|
    \notag \\
    &\le  
    \sum_{\ell \ge 1}^{s} \sum_{\ga \in \Ga(\ell,s)} |c_{\ga,\ell,r}|  
    \, 2^{r-\ell}|\tilde a_2(x_0)|^{r-\ell} \, A^{s}\,|\tilde a_2(x_0)|^{\ell - s/2}
    \notag \\
    &=
    A^{s}\, |\tilde a_2(x_0)|^{r - s/2}
    \sum_{\ell \ge 1}^{s} \sum_{\ga \in \Ga(\ell,s)} |c_{\ga,\ell,r}|\, 2^{r -\ell}.
  \end{align}
  Consequently, by the Leibniz formula, \eqref{eq:A2pt}, and \eqref{eq:A3pt},
  \begin{align*}
      |\ul a_j^{(k)}(x)| \le \sum_{s=0}^{k} \binom{k}{s} |\p^s(|\tilde a_2|^{-j/2})(x)| |\tilde a_j^{(k-s)}(x)| \le C(d)\,A^k\, |\tilde a_2(x_0)|^{ - k/2},
  \end{align*}
  for $1 \le k \le d-1$ and $x \in \II(x_0,A)$, that is \eqref{eq:A5pt}.

  To see \eqref{eq:A6pt} it suffices to repeat the above argument, using that, for functions $f_1,\ldots,f_m$ on an interval $I$, 
  \[
      \Big|\prod_{i=1}^m f_i\Big|_{C^{0,1}(\ol I)} \le \sum_{i=1}^m |f_i|_{C^{0,1}(\ol I)} \prod_{j \ne i} \|f_j\|_{L^{\infty}(I)} 
\]
and 
\begin{equation*}
    |f^{r}|_{C^{0,1}(\ol I)} \le |r|\, \| f^{r-1} \|_{L^\infty(I)} \|f'\|_{L^\infty(I)}, 
\end{equation*}
if $f$ is differentiable.
\end{proof}

\begin{proposition} \label[p]{prop:indpt}
    Let $(\tilde \aa,x_0,\de,A)$ be $C^{d-1,1}$-admissible.
    Then there exist $\de_1>0$ and a constant $C(d)>1$ such that the following holds.
    There is a splitting 
    \[
        P_{\tilde \aa} = P_{\bb} P_{\cc}, \quad \text{ on } I(x_0,\de_1), 
    \]
    where $P_{\bb}$ and $P_{\cc}$ are monic hyperbolic polynomials of degree $<d$ with coefficients in $C^{d-1,1}(\ol I(x_0,\de_1))$. 
    We have, for all $1 \le i \le \deg P_{\bb}$,  
    \begin{align}
        &|b_i^{(d-1)}|_{C^{0,1}(\ol I(x_0,\de_1))} \le C(d) A^d \, |\tilde a_2(x_0)|^{(i-d)/2}. \label{eq:bipt2}
        \\
        &|b_i^{(k)}(x)| \le C(d) A^k \, |\tilde a_2(x_0)|^{(i-k)/2},\quad 1 \le k \le d-1,  ~ x \in I(x_0,\de_1), \label{eq:bipt1}
    \end{align}
    If, after Tschirnhausen transformation, $\tilde b_2(x_0) \ne 0$, then $(\tilde \bb,x_0,\de_1,C(d)A)$ is $C^{d-1,1}$-admissible.
    The analogous statements hold for $\tilde \cc$.
\end{proposition}

\begin{proof}
    Consider the continuous bounded (cf.\ \Cref{lem:dominant} and \Cref{lem:Aconsequence}) curve 
    \[
        \ul \aa := (0,-1,\ul a_3,\ldots,\ul a_d) : \II(x_0,A) \to \on{Hyp}_T^0(d) \subseteq  \R^d,
    \]
    where $\ul a_j := |\tilde a_2|^{-j/2} \tilde a_j$.
    Then, by \eqref{eq:A5pt}, there exists $C_1 =C_1(d)$ such that
    \begin{equation} \label{eq:length}
        \|\ul \aa'(x)\|_2 \le C_1 A\, |\tilde a_2(x_0)|^{-1/2}, \quad x \in \II(x_0,A).
    \end{equation}
    We may assume that $C_1 >1$.
    Let $0<r<1$ be the radius of the splitting (see \Cref{def:univsplit}) and 
    define 
    \[
        \de_1 := \frac{|\tilde a_2(x_0)|^{1/2}\,r}{C_1A}.
    \]
    Then $I(x_0,\de_1) \subseteq  \II(x_0,A)$ and  $\ul \aa(I(x_0,\de_1)) \subseteq B(\ul \aa(x_0),r)$, by \eqref{eq:length}.
    Consequently, we have a splitting
    \[
        P_{\tilde \aa} = P_{\bb} P_{\cc}, \quad \text{ on } I(x_0,\de_1), 
    \]
    by \eqref{eq:radius}.

    Next we check \eqref{eq:bipt2} and \eqref{eq:bipt1}. By \eqref{eq:formulasbhyp},
    \begin{equation} \label{eq:bi}
        b_i = |\tilde a_2|^{i/2} \cdot \ps_i\o \ul \aa. 
    \end{equation}
    We claim that, for $1\le s \le d-1$ and $x \in \ol I(x_0,\de_1)$,
    \begin{equation} \label{eq:psaa}
        |(\ps_i\o \ul \aa)^{(s)}(x)| \le C(d)\, A^s \, |\tilde a_2(x_0)|^{-s/2}.
    \end{equation}
    We have 
    \begin{align*}
        (\ps_i\o \ul \aa)' &= \sum_{j=1}^d ((\p_j \ps_i) \o \ul \aa)\cdot \ul a_j',
        \\
        (\ps_i\o \ul \aa)^{(s)} &= \sum_{j=1}^d \p^{s-1}\big(((\p_j \ps_i) \o \ul \aa)\cdot \ul a_j'\big) 
        \\
                                &= \sum_{j=1}^d \sum_{k=0}^{s-1} \binom{s-1}{k} ((\p_j \ps_i) \o \ul \aa)^{(k)} \ul a_j^{(s-k)}.
    \end{align*}
    For $s=1$ the claim \eqref{eq:psaa} follows from \eqref{eq:A5pt}. For $s \ge 2$ the claim follows by induction and \eqref{eq:A5pt}.

    Now \eqref{eq:bipt1} is a consequence of the Leibniz formula, \eqref{eq:derpower}, \eqref{eq:bi}, and \eqref{eq:psaa}.
    To see \eqref{eq:bipt2} we proceed analogously, combining \eqref{eq:A6pt} with the observations at the end of the proof of \Cref{lem:Aconsequence}.

    Suppose that $\tilde b_2(x_0) \ne 0$ and let us show that $(\tilde \bb,x_0,\de_1,C(d)A)$, 
    for a suitable constant $C(d)>1$, is $C^{d-1,1}$-admissible.
    Set 
    \[
        B:= \frac{2C_1 A}{r}.
    \]
    By \Cref{lem:b2a2hyp}, we have 
    \begin{equation} \label{eq:b2a2}
        |\tilde b_2(x_0)| \le 4 \, |\tilde a_2(x_0)|
    \end{equation}
    which implies
    \[
        B^{-1}|\tilde b_2(x_0)|^{1/2} \le  \frac{|\tilde a_2(x_0)|^{1/2} \,r}{C_1A} = \de_1,
    \]
    so that $\mathbf{J}(x_0,B) := I(x_0,B^{-1}|\tilde b_2(x_0)|^{1/2}) \subseteq I(x_0,\de_1)$.
    From \eqref{eq:bipt2} and \eqref{eq:bipt1}, we easily get 
    the same bounds for $\tilde b_i$ instead of $b_i$ (by means of \eqref{eq:Tschirn}).
    By \eqref{eq:b2a2},
    we may replace $\tilde a_2(x_0)$ by $\tilde b_2(x_0)$ on the right-hand side of these estimates if $k\ge i$ (note that $d>i$).
    Now it suffices to invoke \Cref{lem:enough}.

    The same arguments yield the analogous statement about $\tilde \cc$.
\end{proof}

\begin{proposition} \label[p]{prop:indpt2}
Let $(\tilde \aa,x_0,\de,A)$ be $C^{d-1,1}$-admissible and assume that $x_0$ is not reclusive for $\tilde \aa$.
If $\la \in C^0(I(x_0,\de))$ is a root of $P_{\tilde \aa}$ and $\la'(x_0)$ exists, 
    then 
    \begin{equation} \label{eq:lapr}
        |\la'(x_0)| \le C(d) A.
    \end{equation}
\end{proposition}

\begin{proof}
    By assumption, $\tilde a_2(x_0) \ne 0$ and hence $d\ge 2$.
    By \Cref{prop:indpt}, there exists $\de_1>0$ such that there is a splitting 
    $P_{\tilde \aa} = P_{\bb}P_{\cc}$ on $I(x_0,\de_1)$.
    We may assume that $\la$ is a root of $P_{\bb}$ and hence 
    \begin{equation} \label{eq:rootTschirn}
        \la(x) = -\frac{b_1(x)}{\deg P_{\bb}} + \mu(x), \quad x \in I(x_0,\de_1),
    \end{equation}
where $\mu$ is a continuous root of $P_{\tilde \bb}$ and $\mu'(x_0)$ exists 
    (since we assumed that $\la'(x_0)$ exists). 
    By \eqref{eq:bipt1} for $i=k=1$, we have 
    \begin{equation} \label{eq:b1}
        |b_1'(x_0)| \le C(d) A.
    \end{equation}
    By \Cref{lem:splitrec}, $x_0$ is not reclusive for $\tilde \bb$, since $x_0$ is not reclusive for $\tilde \aa$. 
    
    Let us now prove the proposition by induction on $d$. 
        
    If $d=2$, then $\deg P_{\bb} =1$ and $\tilde \bb \equiv 0$. Thus $\la(x)=-b_1(x)$ for $x \in I(x_0,\de_1)$ so that \eqref{eq:b1} gives \eqref{eq:lapr}. 

    Assume that $d\ge 3$.
    If $\tilde b_2(x_0) \ne 0$, then $(\tilde \bb,x_0,\de_1,C(d)A)$ is $C^{d-1,1}$-admissible, by \Cref{prop:indpt}.
    By the induction hypothesis, 
    \[
        |\mu'(x_0)| \le C(d) A.
    \]
    Thus \eqref{eq:lapr} follows from \eqref{eq:rootTschirn} and \eqref{eq:b1}.

    If $\tilde b_2(x_0) = 0$, then $x_0$ (being not reclusive for $\tilde \bb$) is an accumulation point of $Z_{\tilde \bb}$.
    Consequently, 
    $\mu'(x_0)=0$
    and \eqref{eq:lapr} again follows from \eqref{eq:rootTschirn} and \eqref{eq:b1}.
\end{proof}

\begin{proof}[Proof of \Cref{thm:ptwkey}]
    Let $x_0 \in \R$ and $\de>0$.
Let $\tilde \aa \in C^{d-1,1}(\ol I(x_0,\de),\on{Hyp}_T(d))$.
Assume that $x_0$ is not reclusive for $\tilde \aa$.
Let $\la \in C^0(I(x_0,\de))$ be a continuous root of $P_{\tilde \aa}$ and assume that $\la'(x_0)$ exists.
   
    If $\tilde a_2(x_0) \ne 0$, then, by \Cref{lem:ptwass>adm}, 
    $(\tilde \aa,x_0,\de,C(d)A(\de))$ is $C^{d-1,1}$-admissible, where $A(\de)$ is defined by \eqref{eq:defAde} and $C(d)\ge 1$.
    Then \Cref{prop:indpt2} yields 
    \[
       |\la'(x_0)| \le C(d) A(\de). 
    \]
If $\tilde a_2(x_0) = 0$, then $x_0$ (being not reclusive for $\tilde \aa$) is an accumulation point of $Z_{\tilde \aa}$. 
    Hence $\la'(x_0) =0$ and the assertion is trivially true.
\end{proof}

\section{Proof of the main technical result} \label{sec:proof1}

The goal of this section is the proof of \Cref{thm:mainhyp}.

Let $I \subseteq \R$ be an open interval.
    Let $\aa_n \to \aa$ in $C^d (I, \on{Hyp}(d))$, 
    i.e.,  
    for each relatively compact open interval $I_1 \Subset I$,
    \begin{equation*}
        \|\aa- \aa_n\|_{C^{d}(\ol I_1,\R^d)} \to 0 \quad \text{ as } n \to \infty.
    \end{equation*}
    It follows from \Cref{thm:Bronshtein} that the set 
    $\{\sol(\aa_n) : n \ge 1\}$ is bounded in $C^{0,1}(I,\R^d)$.
    We must show that, for each relatively compact open interval $I_0 \Subset I$ and each $1 \le q < \infty$,
    \begin{equation} \label{eq:W1q}
        \|\sol(\aa) -  \sol(\aa_n)\|_{W^{1,q}(I_0,\R^d)}  \to 0 \quad \text{ as } n\to \infty. 
    \end{equation}

\subsection{Strategy of the proof} \label{ssec:strategy}

The proof of \eqref{eq:W1q} is subdivided into three steps. The first two steps are dedicated to the proof of
\begin{equation}     \label{eq:mainconclder}       
    \|\sol(\aa)' -  \sol(\aa_n)'\|_{L^{q}(I_0,\R^d)}  \to 0 \quad \text{ as } n\to \infty, 
\end{equation}
for $1 \le q < \infty$, 
using the dominated convergence theorem. In the third step, we show that 
\begin{equation} \label{eq:C0conv}
    \|\sol(\aa) -  \sol(\aa_n)\|_{L^{\infty}(I_0,\R^d)}  \to 0 \quad \text{ as } n\to \infty. 
\end{equation}
Then \eqref{eq:mainconclder} and \eqref{eq:C0conv} imply \eqref{eq:W1q}.

\subsection*{Step 1}

We claim that the sequence of derivatives $\sol(\aa_n)'$ is dominated almost everywhere on $I_0$ by a positive constant.

To see this, fix $I_0 \Subset I_1 \Subset I$. 
By the assumption of \Cref{thm:mainhyp},
$\{\aa_n|_{I_1} : n \ge 1\}$ is a bounded subset of $C^{d-1,1}(\ol I_1,\R^d)$. 
By \Cref{thm:Bronshtein}, the derivative of $\sol(\aa_n)$ exists almost everywhere in $I_0$ and 
satisfies 
\[
    \|\sol(\aa_n)'\|_{L^\infty(I_0,\R^d)}\le C\, \sup_{n\ge 1} \max_{1 \le j \le d} \|a_{n,j}\|^{1/j}_{C^{d-1,1}(\ol I_1)}=: B < \infty.
\]
This implies the claim.

\subsection*{Step 2}

The aim of this step is to prove the following result.

\begin{theorem} \label[t]{thm:ptwhyp}
    Let $I \subseteq \R$ be an open interval.
    Let $\aa_n \to \aa$ in $C^d (I, \on{Hyp}(d))$ as $n \to \infty$.
    Then, for almost every $x \in I$, 
    \[
        \sol(\aa_n)'(x) \to \sol(\aa)'(x) \quad \text{ as } n \to \infty.  
    \]
\end{theorem}

By Step 1 and 2, the dominated convergence theorem yields that 
\eqref{eq:mainconclder} holds, for each relatively compact open interval $I_0 \Subset I$ and
each $1 \le q <\infty$.  

\subsection*{Step 3}

Now we show \eqref{eq:C0conv}.
Fix $x_0 \in I_0$. We have 
\begin{equation} \label{eq:ptw}
    \sol(\aa_n)(x_0)=\la^\uparrow(\aa_n(x_0))  \to \la^\uparrow(\aa(x_0))=\sol(\aa)(x_0) \quad \text{ as } n \to \infty,
\end{equation}
since the map $\la^\uparrow : \on{Hyp}(d) \to \R^d$ is continuous (cf.\ \cite[Lemma 4.1]{AKLM98}). 
For arbitrary $x \in I_0$, 
\begin{align*}
    \|\sol(\aa)(x) - \sol(\aa_n)(x)\|_2 &= \Big\|\sol(\aa)(x_0) - \sol(\aa_n)(x_0) + \int_{x_0}^x \sol(\aa)'(t) - \sol(\aa_n)'(t) \,dt \Big\|_2  
    \\
                                        &\le \|\sol(\aa)(x_0) - \sol(\aa_n)(x_0)\|_2 + \|\sol(\aa)'-\sol(\aa_n)'\|_{L^1(I_0,\R^d)}.
\end{align*}
Thus, \eqref{eq:mainconclder} and \eqref{eq:ptw} imply \eqref{eq:C0conv}. 

\begin{remark}
    Alternatively, \eqref{eq:C0conv} is a consequence of
    \cite[Corollary 6.5]{Parusinski:2024ab} and \Cref{lem:unordered}; 
    for this argument
    it is actually enough that $\aa_n \to \aa$ in $C^0(\ol I_0,\on{Hyp}(d))$ 
    as $n \to \infty$.
\end{remark}

Therefore, in order to prove \Cref{thm:mainhyp} 
it remains to show \Cref{thm:ptwhyp}. 
The rest of the section is devoted to the proof of \Cref{thm:ptwhyp}.

\subsection{On the zero set of $\tilde \aa$}

Recall that $\on{acc}(Z_{\tilde \aa})$ denotes the set of accumulation points of the zero set $Z_{\tilde \aa}$ 
of $\tilde \aa : I \to \Hyp_T(d)$. By \Cref{lem:dominant}, $Z_{\tilde \aa} = Z_{\tilde a_2}$.

\begin{lemma} \label[l]{lem:zeroset}
    Let $I \subseteq \R$ be a bounded open interval.
    Let $\tilde \aa_n \to \tilde \aa$ in $C^d(I, \on{Hyp}_T(d))$ as $n\to \infty$.
    Then, for almost every $x_0 \in Z_{\tilde \aa}$, 
    \begin{equation} \label{eq:to0}
    \sol(\tilde \aa_n)'(x_0) \to 0 \quad \text{ as } n \to \infty.  
    \end{equation}
\end{lemma}

\begin{proof}
    We will show that \eqref{eq:to0} holds for all $x_0 \in J$, where
    \begin{align*}
        J :=  \on{acc}(Z_{\tilde \aa}) &\cap \bigcap_{n\ge 1}\{x \in I : x \text{ is not reclusive for } \tilde \aa_n\} 
        \\
                                       &\cap \bigcap_{n \ge 1} \{x \in I : \sol(\tilde \aa_n)'(x) \text{ exists} \}.
    \end{align*}
    The set $J$ has full measure in $Z_{\tilde \aa}$, by \Cref{lem:countable} and Rademacher's theorem.

    Fix $x_0 \in J \subseteq \on{acc}(Z_{\tilde \aa})$.
    Then $\tilde a_j^{(k)}(x_0) = 0$ for all $2 \le j \le d$ and $0 \le k \le d$, by Rolle's theorem.
    Let $\ep>0$ be fixed.
    By continuity, there exists $\de>0$ such that $I(x_0,\de) \Subset I$ 
    and 
    \begin{equation} \label{eq:conta}
        \|\tilde a_j\|_{C^d(\ol I(x_0,\de))} \le \frac{\ep^j}{2}, \quad 2 \le j \le d.
    \end{equation}
    Since $\tilde \aa_n \to \tilde \aa$ in $C^d(I, \on{Hyp}_T(d))$ as $n\to \infty$, there exists $n_0 \ge 1$ such that, for $n \ge n_0$, 
    \begin{equation} \label{eq:ana}
        \|\tilde a_j - \tilde a_{n,j}\|_{C^d(\ol I(x_0,\de))} \le \frac{\ep^j}{2}, \quad 2 \le j \le d,
    \end{equation}
    and 
    \begin{equation} \label{eq:an2x0}
        |\tilde a_{n,2}(x_0)| \le \de^2 \ep^2. 
    \end{equation}
    By \eqref{eq:conta} and \eqref{eq:ana}, for $n \ge n_0$ and $2 \le j \le d$,
    \begin{equation} \label{eq:contan}
        \|\tilde a_{n,j}\|_{C^d(\ol I(x_0,\de))} 
        \le \|\tilde a_j\|_{C^d(\ol I(x_0,\de))}+\|\tilde a_j - \tilde a_{n,j}\|_{C^d(\ol I(x_0,\de))} 
        \le \ep^j.
    \end{equation}

    Since $x_0 \in J$ is not reclusive for $\tilde \aa_n$ and $\sol(\tilde \aa_n)'(x_0)$ exists, 
    we may apply \Cref{thm:ptwkey} to $\tilde \aa_n$ and conclude
    \[
        \|\sol(\tilde \aa_n)'(x_0)\|_2 \le C(d) \, A(\de),
    \]
    where $A(\de)$ is defined in \eqref{eq:defAde} with $\tilde \aa$ replaced by $\tilde \aa_n$. 
    By \eqref{eq:an2x0} and \eqref{eq:contan}, 
    \[
        A(\de) \le 6\, \ep.
    \]
    Since $\ep>0$ was arbitrary, we conclude that 
    \[
        \sol(\tilde \aa_n)'(x_0) \to 0 \quad \text{ as } n \to \infty.  
    \]
    The proof is complete.
\end{proof}

\subsection{Admissible data}

At points $x_0$ with $\tilde a_2(x_0)\ne 0$ we have a splitting of $P_{\tilde \aa}$ and we 
may use induction on the degree. The following definition is a preparation for the 
induction argument.

Let us recall (from \eqref{eq:II}) that 
\[
        \II(x_0,A) := I(x_0,A^{-1}|\tilde a_2(x_0)|^{1/2}).
\]

\begin{definition}[$C^d$-admissible data] \label[d]{def:adm}
    Let $I_1 \subseteq \R$ be an open bounded interval and $I_0 \Subset I_1$ a relatively compact open subinterval.
    Let $\tilde \aa \in C^{d}(\ol I_1,\on{Hyp}_T(d))$.
    Let $A>0$ be a constant. 
    We say that $(\tilde \aa,I_0,I_1,A)$ is \emph{$C^d$-admissible} if, for 
    every $x_0 \in I_0 \setminus \{x : \tilde a_2(x) = 0\}$, the following holds:
    \begin{enumerate}
        \item $\II(x_0,A)  \subseteq I_1$. \label{eq:A1}
        \item For all $x \in \II(x_0,A)$,
            \begin{equation}
                \frac{1}{2} \le \frac{\tilde a_2(x)}{\tilde a_2(x_0)} \le 2.  \label{eq:A2}
            \end{equation}
        \item For all $2\le j \le d$, $1 \le k \le d$, and $x \in \II(x_0,A)$,
            \begin{equation}
                |\tilde a_j^{(k)}(x)| \le  A^k \, |\tilde a_2(x_0)|^{(j-k)/2}. \label{eq:A3}
            \end{equation}
    \end{enumerate}
\end{definition}

Note that if we take $I_1:= I(x_0,\de)$, let $I_0$ shrink to the point $x_0$, assume $\tilde a_2(x_0)\ne 0$, and 
use $C^{d-1,1}$- instead of $C^d$-regularity, we 
recover the notion from \Cref{def:ptadm}.

\begin{lemma} \label[l]{lem:ass>adm}
    Let $I_1 \subseteq \R$ be a bounded open interval
    and $I_0 \Subset I_1$ a relatively compact open subinterval.
    Let $\tilde \aa_n \to \tilde \aa$ in $C^d (\ol I_1, \on{Hyp}_T(d))$ as $n\to \infty$. 
    Set
    \begin{align}
        A&:= 6 \max\{A_1,A_2\}, \label{eq:defA} 
        \intertext{where, using $\tilde a_{0,j} = \tilde a_j$ for convenience
        and $\de := \on{dist}(I_0, \R \setminus I_1)$,} 
        A_1 &:=\sup_{n\ge 0} \max \big\{ \de^{-1}\|\tilde a_{n,2}\|_{L^\infty(I_1)}^{1/2}, |\tilde a_{n,2}'|_{C^{0,1}(\ol I_1)}^{1/2}\big\}, \nonumber
        \\
        A_2 &:=\sup_{n \ge 0} \max_{2 \le j \le d} \big\{ \|\tilde a_{n,j}^{(d)}\|_{L^\infty(I_1)} \cdot \|\tilde a_{n,2}\|_{L^\infty(I_1)}^{(d-j)/2}  \big\}^{1/d}. \nonumber
    \end{align}
    Then $(\tilde \aa, I_0,I_1,C(d)A)$ and $(\tilde \aa_n, I_0,I_1,C(d)A)$, for $n \ge 1$, are $C^d$-admissible,
    for some constant $C(d)\ge 1$.
\end{lemma}

\begin{proof}
    Fix $n\ge 0$ and $x_0 \in I_0 \setminus \{x : \tilde a_{n,2}(x) = 0\}$. By the definition of $A_1$, we have $\II(x_0,A_1) \subseteq I_1$.
    By \Cref{lem:Glaeser},
    \[
        |\tilde a_{n,2}'(x_0)| \le 2 A_1\, |\tilde a_{n,2}(x_0)|^{1/2}
    \]
    which entails (as in the proof of \Cref{lem:enough}) that \eqref{eq:A2} holds on $\II(x_0,6A_1)$. 
    By the definition of $A_2$, \eqref{eq:A3} holds for $k=d$. 
    By \Cref{lem:dominant} and \eqref{eq:A2}, we have 
     $|\tilde a_j(x)| \le 
     2^j\,|\tilde a_2(x_0)|^{j/2}$, for $x \in \II(x_0,A)$. Thus, \eqref{eq:A3} follows from
    \Cref{lem:interpol}. 
\end{proof}

In the following, we will use $\II(x_0,A)$ as well as its counterpart for $\aa_n$ instead of $\aa$, 
that is
\begin{equation} \label{eq:IIn}
    \II_n(x_0,A) := I(x_0,A^{-1}|\tilde a_{n,2}(x_0)|^{1/2}).
\end{equation}

\subsection{Towards a simultaneous splitting}

Our next goal is to show that, if $\tilde \aa_n \to \tilde \aa$ in $C^d (\ol I_1, \on{Hyp}_T(d))$ as $n \to \infty$
and
$(\tilde \aa, I_0,I_1,A)$ and $(\tilde \aa_n, I_0,I_1,A)$, for $n \ge 1$, are $C^d$-admissible, then 
$P_{\tilde \aa}$ and $P_{\tilde \aa_n}$, for $n$ large enough, 
admit a simultaneous splitting; see \Cref{def:simsplit}.

    Let $I_1 \subseteq \R$ be a bounded open interval
    and $I_0 \Subset I_1$ a relatively compact open subinterval.
    Let $\tilde \aa_n \to \tilde \aa$ in $C^d (\ol I_1, \on{Hyp}_T(d))$, i.e.,
    \begin{equation} \label{eq:Cdconvergence}
        \|\tilde \aa-\tilde \aa_n\|_{C^{d}(\ol I_1,\R^d)} \to 0 \quad \text{ as } n \to \infty.
    \end{equation}
Assume that 
$(\tilde \aa, I_0,I_1,A)$ and $(\tilde \aa_n, I_0,I_1,A)$, for $n \ge 1$, are $C^d$-admissible
for some $A>0$.

Fix $x_0 \in I_0 \setminus \{x : \tilde a_2(x)=0\}$.
By \eqref{eq:Cdconvergence}, 
there is $n_0 \ge 1$ such that 
\begin{equation*} 
    ||\tilde a_{2}(x_0)|^{1/2}-|\tilde a_{n,2}(x_0)|^{1/2}| < \frac{1}{2}\, |\tilde a_{2}(x_0)|^{1/2},
    \quad n \ge n_0,
\end{equation*}
and hence
\begin{equation}\label{eq:2close}
    \frac{1}2 < \frac{|\tilde a_{n,2}(x_0)|^{1/2}}{|\tilde a_{2}(x_0)|^{1/2}}<\frac{3}{2}, 
    \quad n \ge n_0.
\end{equation}
So, for $n \ge n_0$,
\begin{equation} \label{eq:Iu}
    \II(x_0,2A) \subseteq \II_n(x_0,A) \subseteq \II(x_0,2A/3),
\end{equation}
where $\II_n(x_0,A)$ is defined in \eqref{eq:IIn}.
Since $(\tilde \aa_n, I_0,I_1,A)$, for $n \ge 1$, is $C^d$-admissible
and thanks to \eqref{eq:Iu} we see that, for $n \ge n_0$,
\begin{align}
        &\II(x_0,2A) \subseteq I_1, \label{eq:uA1}  
        \\
        &\frac{1}{2} \le \frac{\tilde a_{n,2}(x)}{\tilde a_{n,2}(x_0)} \le 2,\quad x \in \II(x_0,2A),  \label{eq:uA2}
        \\
        &|\tilde a_{n,j}^{(k)}(x)| \le  A^k \, |\tilde a_{n,2}(x_0)|^{(j-k)/2},\quad 2\le j \le d,\, 1 \le k \le d,\,   x \in \II(x_0,2A). \label{eq:uA3}
    \end{align}

    Consider the $C^{d}$ curves 
    \begin{align*}
        \ul \aa &:= (0,-1,\ul a_3,\ldots,\ul a_d) : \II(x_0,2A) \to \on{Hyp}_T^0(d) \subseteq \R^d,
        \\
        \ul \aa_n &:= (0,-1,\ul a_{n,3},\ldots,\ul a_{n,d}) : \II(x_0,2A) \to \on{Hyp}_T^0(d) \subseteq \R^d, \quad n\ge n_0,
    \end{align*}
    where $\ul a_j := |\tilde a_2|^{-j/2} \tilde a_j$ and $\ul a_{n,j} := |\tilde a_{n,2}|^{-j/2} \tilde a_{n,j}$.
    Then, by the proof of \Cref{lem:Aconsequence}, we conclude that
    there is a constant
    \begin{equation} \label{eq:C1}
        C_1 = C_1(d) >1
    \end{equation}
    such that, 
    for $x \in \II(x_0,2A)$, 
    \begin{align*}
        \|\ul \aa'(x)\|_2 \le C_1 A \, |\tilde a_2(x_0)|^{-1/2} \quad \text{ and } \quad  \|\ul \aa_n'(x)\|_2 \le C_1 A\, |\tilde a_{n,2}(x_0)|^{-1/2}.
    \end{align*}
    Let $0<r<1$ be the radius of the splitting (see \Cref{def:univsplit}) and 
    define 
    \[
        J_1 := \II(x_0,4C_1A/r)= I(x_0,\tfrac{r}{4C_1 A}|\tilde a_2(x_0)|^{1/2}).
    \]
    Then $\ul \aa(J_1) \subseteq B(\ul \aa(x_0),r/4)$ and $\ul \aa_n(J_1) \subseteq B(\ul \aa_n(x_0),r/2)$, using \eqref{eq:2close}. 
    By \eqref{eq:Cdconvergence},
    there is $n_1 \ge n_0$ such that
    \begin{equation}
        \|\ul \aa(x_0) - \ul \aa_n(x_0)\|_2 < \frac{r}{4}, \quad n \ge n_1.
    \end{equation}
    Consequently, $B(\ul \aa_n(x_0),r/2)$ is contained in $B(\ul \aa(x_0),3r/4)$, for $n \ge n_1$. 

    In view of \eqref{eq:radius} and \Cref{def:univsplit},
we have splittings on $J_1$,
\begin{equation} \label{eq:simsplit}
    P_{\tilde \aa} = P_{\bb} P_{\cc} \quad \text{ and } \quad  P_{\tilde \aa_n} = P_{\bb_n} P_{\cc_n}, \quad n\ge n_1,
\end{equation}
with the following properties:
\begin{enumerate}
    \item $d_{\bb} := \deg P_{\bb} = \deg P_{\bb_n}$, for all $n\ge n_1$, and $d_{\bb} < d$.
    \item There exist bounded analytic functions $\ps_1,\ldots,\ps_{d_{\bb}}$ with bounded partial derivatives of all orders
        such that, for $x \in J_1$ and $1 \le i \le d_{\bb}$,
    \begin{align*}
        b_i(x) &= |\tilde a_2(x)|^{i/2} \, \ps_i(\ul \aa(x) ), 
        \\
        b_{n,i}(x) &= |\tilde a_{n,2}(x)|^{i/2}\, \ps_i(\ul \aa_n(x) ), \quad n\ge n_1.
    \end{align*}
\end{enumerate}
The same is true for the second factors $P_{\cc}$ and $P_{\cc_n}$.
    
\begin{definition}[Simultaneous splitting] \label[d]{def:simsplit}
    We say that the family $\{P_{\tilde \aa}\} \cup \{P_{\tilde \aa_n}\}_{n\ge n_1}$ has a 
    \emph{simultaneous splitting on an interval $J_1$} if \eqref{eq:simsplit}
    and the above properties (1) and (2) are satisfied.
    \end{definition}

    Note that, applying the Tschirnhausen transformation to $P_{\bb}$ and $P_{\bb_n}$ and by \eqref{eq:Tschirn},
    we find bounded analytic functions $\tilde \ps_1,\ldots,\tilde \ps_{d_{\bb}}$ with bounded partial derivatives of all orders
    such that, for $x \in J_1$ and $1 \le i \le d_{\bb}$,
    \begin{align*}
        \tilde b_i(x) &= |\tilde a_2(x)|^{i/2} \, \tilde \ps_i(\ul \aa(x) ), 
        \\
        \tilde b_{n,i}(x) &= |\tilde a_{n,2}(x)|^{i/2}\, \tilde  \ps_i(\ul \aa_n(x) ), \quad n\ge n_1.
    \end{align*}

    \begin{lemma} \label[l]{lem:bconv} 
        We have $\bb_n \to \bb$ and $\tilde \bb_n \to \tilde \bb$ in $C^{d}(\ol J_1,\R^{d_{\bb}})$ as $n \to \infty$. 
    \end{lemma}

    \begin{proof}
        By \eqref{eq:A2} and \eqref{eq:uA2}, $|\tilde a_2|^{1/2},  |\tilde a_{n,2}|^{1/2} \in C^d(\ol J_1)$
        and
        $\ul \aa, \ul \aa_n \in C^{d}(\ol J_1,\R^d)$, for $n\ge n_0$, and
        the assertion follows from \Cref{prop:lefttr}.
    \end{proof}

Summarizing, we have the following proposition.

\begin{proposition} \label[p]{prop:aftersplit}
    Let $I_1 \subseteq \R$ be a bounded open interval
    and $I_0 \Subset I_1$ a relatively compact open subinterval.
    Let $\tilde \aa_n \to \tilde \aa$ in $C^d (\ol I_1, \on{Hyp}_T(d))$ as $n\to \infty$. 
    Assume that 
    $(\tilde \aa, I_0,I_1,A)$ and $(\tilde \aa_n, I_0,I_1,A)$, for $n \ge 1$, are $C^d$-admissible
    for some $A>0$. Let $x_0 \in I_0 \setminus \{x : \tilde a_2(x)=0\}$.
    Then the following holds:
    \begin{enumerate}
        \item There exist an interval $J_1$ containing $x_0$ and $n_0\ge 1$ such that
            the family $\{P_{\tilde \aa}\} \cup \{P_{\tilde \aa_n}\}_{n\ge n_0}$ has a 
            simultaneous splitting \eqref{eq:simsplit} on $J_1$.
        \item For the factors in the simultaneous splitting  \eqref{eq:simsplit},
            $\bb_n \to \bb$ and $\tilde \bb_n \to \tilde \bb$ in $C^{d}(\ol J_1,\R^{d_{\bb}})$
             as $n \to \infty$. 
        \item There exist a relatively compact open subinterval $J_0 \Subset J_1$ containing $x_0$ and 
            a constant $C=C(d) >1$ such that
            $(\tilde {\bb}, J_0,J_1,CA)$ 
            and $(\tilde {\bb}_n, J_0,J_1,CA)$, for $n \ge n_0$, are $C^d$-admissible. 
    \end{enumerate}
    The properties \thetag{2} and \thetag{3} also hold for $\bb,\bb_n,\tilde \bb, \tilde \bb_n$  
    replaced by $\cc,\cc_n,\tilde \cc, \tilde \cc_n$. 
\end{proposition}

\begin{proof}
    (1) This was proved above.

    (2) \Cref{lem:bconv}.

    (3) 
    Set $J_0 := \II(x_0,8C_1A/r) = I(x_0, \frac{r}{8C_1A} |\tilde a_2(x_0)|^{1/2})$,  
    where $C_1$ is the constant from \eqref{eq:C1}. Then clearly $J_0 \Subset J_1$.
    
    Fix $x_1 \in J_0 \setminus \{x : \tilde b_2(x) = 0\}$. Then 
    \[
        |x_1-x_0| < \frac{r}{8C_1A} |\tilde a_2(x_0)|^{1/2}.
    \]
    By \Cref{lem:b2a2hyp} and \eqref{eq:A2},
    \[
         |\tilde b_2(x_1)|^{1/2} \le 2 \, |\tilde a_2(x_1)|^{1/2} \le 2\sqrt{2}\,  |\tilde a_2(x_0)|^{1/2}.
    \]
    Setting 
    \[
        B := \frac{16 \sqrt 2\, C_1A}{r}
    \]
    we have
    \[
        B^{-1} \, |\tilde b_2(x_1)|^{1/2} \le \frac{r}{8C_1A}\, |\tilde a_2(x_0)|^{1/2},
    \]
    and hence 
    \[
        \mathbf{J}(x_1,B) := I(x_1,B^{-1} \, |\tilde b_2(x_1)|^{1/2} ) \subseteq I(x_0,\tfrac{r}{4C_1A} |\tilde a_2(x_0)|^{1/2}) = J_1.
    \]
    
    One checks, exactly as in the proof of \Cref{prop:indpt}, that 
            \begin{equation*}
                |\tilde b_i^{(k)}(x)| \le C(d)\, A^k \, |\tilde a_2(x_0)|^{(i-k)/2},
            \end{equation*}
    for all $2 \le i \le d_\bb$, $1 \le k \le d$, and $x \in J_1$.
    If $k\ge i$, we may replace $\tilde a_2(x_0)$ by $\tilde b_2(x_1)$ on the right-hand side. Thus,
    we may conclude that $(\tilde {\bb}, J_0,J_1,CA)$ is $C^d$-admissible, for a suitable constant $C=C(d)>1$, 
    by the proof of \Cref{lem:enough}.

    To see that also $(\tilde {\bb}_n, J_0,J_1,CA)$ is $C^d$-admissible, fix $x_1 \in J_0 \setminus \{x : \tilde b_{n,2}(x) = 0\}$. 
    By \Cref{lem:b2a2hyp}, \eqref{eq:2close}, and \eqref{eq:uA2}, 
    \[
        |\tilde b_{n,2}(x_1)|^{1/2} \le 2 \, |\tilde a_{n,2}(x_1)|^{1/2} \le 2\sqrt{2}\,  |\tilde a_{n,2}(x_0)|^{1/2} \le 3\sqrt{2}\,  |\tilde a_{2}(x_0)|^{1/2}.
    \]
    Hence, using 
    \[
        B := \frac{24 \sqrt 2\, C_1A}{r},
    \]
    we find
    \[
        \mathbf{J}_n(x_1,B) := I(x_1,B^{-1} \, |\tilde b_{n,2}(x_1)|^{1/2} ) \subseteq I(x_0,\tfrac{r}{4C_1A} |\tilde a_2(x_0)|^{1/2}) = J_1.
    \]
    The rest follows in the same way as described above.
\end{proof}

\subsection{The induction argument}

    \begin{proposition} \label[p]{prop:induction}
        Let $I_1 \subseteq \R$ be a bounded open interval
        and $I_0 \Subset I_1$ a relatively compact open subinterval.
        Let $\tilde \aa_n \to \tilde \aa$ in $C^d (\ol I_1, \on{Hyp}_T(d))$ as $n\to \infty$. 
    Assume that 
    $(\tilde \aa, I_0,I_1,A)$ and $(\tilde \aa_n, I_0,I_1,A)$, for $n \ge 1$, are $C^d$-admissible
    for some $A>0$. 
        Then, for almost every $x \in I_0$,
        \begin{equation} \label{eq:ptcind}
            \sol(\tilde \aa_n)'(x) \to \sol(\tilde \aa)'(x)  \quad \text{ as } n\to \infty. 
        \end{equation}
    \end{proposition}

    \begin{proof}
        We proceed by induction on $d$.
        The base case is trivial, since $Z$ is the unique polynomial in Tschirnhausen form of degree $1$.
        Let us assume that $d\ge 2$ and that the statement is true for monic hyperbolic polynomials of degree $\le d-1$.

        If $x \in \on{acc}(Z_{\tilde a})$ and $\sol(\tilde \aa)'(x)$ exists, then $\sol(\tilde \aa)'(x)=0$.
        Thus, by \Cref{lem:zeroset}, it is enough to show that \eqref{eq:ptcind} holds for almost every $x \in I_0 \setminus \{x: \tilde a_2(x)=0\}$. 
        
        Fix $x_0 \in I_0 \setminus \{x : \tilde a_2(x)=0\}$.
        By \Cref{prop:aftersplit}, 
        there exist intervals $J_1 \Supset J_0 \ni x_0$,  $n_0\ge 1$, and $C=C(d)>1$ such that 
        the family $\{P_{\tilde \aa}\} \cup \{P_{\tilde \aa_n}\}_{n\ge n_0}$ has a 
            simultaneous splitting \eqref{eq:simsplit} on $J_1$, 
            $(\tilde \bb, J_0,J_1,CA)$ 
            and $(\tilde \bb_n, J_0,J_1,CA)$, for $n \ge n_0$, are $C^d$-admissible, and
            $\bb_n \to \bb$ and $\tilde \bb_n \to \tilde \bb$ in $C^{d}(\ol J_1,\R^{d_{\bb}})$ as $n \to \infty$.

        We may assume that, for $x \in J_1$, 
        \[
            \mu(x):= (\la^\uparrow_1(\tilde \aa(x)),\la^\uparrow_2(\tilde \aa(x)), \ldots,\la^\uparrow_{d_{\bb}}(\tilde \aa(x)))
        \]
    is the increasingly ordered root vector of $P_{\bb(x)}$ and, for $n \ge n_0$, 
        \[
            \mu_n(x):=  (\la^\uparrow_1(\tilde \aa_n(x)),\la^\uparrow_2(\tilde \aa_n(x)), \ldots,\la^\uparrow_{d_{\bb}}(\tilde \aa_n(x)))
        \]
        is the increasingly ordered root vector of $P_{\bb_n(x)}$; see \Cref{def:univsplit}. 
        Then
        \[
        \mu(x) + \tfrac{1}{d_{\bb}}(b_1(x),\ldots,b_1(x))  \quad \text{ and } \quad \mu_n(x) + \tfrac{1}{d_{\bb}}(b_{n,1}(x),\ldots,b_{n,1}(x))
        \]
        are the corresponding root vectors for $P_{\tilde \bb(x)}$ and $P_{\tilde \bb_n(x)}$, respectively.
        By induction hypothesis and since $b_{n,1}'(x) \to b_1'(x)$ as $n \to \infty$, we have 
        \[
            \mu_n'(x) \to \mu'(x)  \quad \text{ as } n\to \infty,
        \]
        for almost every $x \in J_0$.
        
        Treating the second factors $P_{\cc}$ and $P_{\cc_n}$ analogously, we conclude that
        \eqref{eq:ptcind} holds
        for almost every $x \in J_0$.
        
        The set $I_0 \setminus \{x : \tilde a_2(x)=0\}$ can be covered by the open intervals $J_0$ 
        and this cover admits a countable subcover.
        This ends the proof.
    \end{proof}

\subsection{Proof of \Cref{thm:ptwhyp}}

Let $I \subseteq \R$ be an open interval.
Let $\aa_n \to \aa$ in $C^d (I, \on{Hyp}(d))$ as $n\to \infty$.
The Tschirnhausen transformation effects a shift of $\sol(\aa)$ by $\frac{1}{d}(a_1,\ldots,a_1)$ 
and of $\sol(\aa_n)$ by $\frac{1}{d}(a_{n,1},\ldots,a_{n,1})$.
The new coefficients are polynomials in the old ones, see \eqref{eq:Tschirn}.
Hence we may assume that the polynomials are all in Tschirnhausen form (by \Cref{prop:lefttr}).
Then \Cref{thm:ptwhyp} follows from \Cref{lem:ass>adm} and \Cref{prop:induction}.

This also completes the proof of \Cref{thm:mainhyp}.

\begin{remark}
    We need $C^d$ convergence in \Cref{lem:zeroset}.
    For all other arguments, it would be enough to work in the class $C^{d-1,1}$. 
\end{remark}

\section{Proofs of the main results} \label{sec:proof2}

In this section, we will deduce \Cref{thm:solmap} and \Cref{cor:solmap} from \Cref{thm:mainhyp}.
We also prove \Cref{t:simple}.

\subsection{A multiparameter version}

The following theorem is a multiparameter version of \Cref{thm:mainhyp}.

\begin{theorem} \label[t]{thm:multhyp}
    Let $U \subseteq \R^m$ be open.
    Let $\aa_n \to \aa$ in $C^d(U,\on{Hyp}(d))$,
    i.e.,  
    for each relatively compact open subset $U_1 \Subset U$,
    \begin{equation*} 
        \| \aa- \aa_n\|_{C^{d}(\ol U_1,\R^d)} \to 0 \quad \text{ as } n \to \infty.
    \end{equation*}
    Then $\{\sol(\aa_n) : n \ge 1\}$ is a bounded set in $C^{0,1}(U, \R^d)$
    and, for each relatively compact open subset $U_0 \Subset U$ and each $1 \le q < \infty$,
    \[
        \|\sol(\aa) -  \sol(\aa_n)\|_{W^{1,q}(U_0,\R^d)}  \to 0 \quad \text{ as } n\to \infty. 
    \]
\end{theorem}

\begin{proof}
    Let us assume that $U_0$ is an open box $U_0 = I_{1} \times \cdots \times I_{m}$
    with sides parallel to the coordinate axes. 
    Set $\la := \sol(\aa)$ and $\la_n := \sol(\aa_n)$. Let $x = (x_1,x')$ and for 
    $x' \in U_0'= I_{2} \times \cdots \times I_{m}$ consider
    \[
        A_n(x') := \int_{I_{1}} \|\p_1 \la(x_1,x') - \p_1 \la_n(x_1,x')\|_2^q\, dx_1. 
    \]
    Then $A_n(x') \to 0$ as $n \to \infty$, by \Cref{thm:mainhyp}. 
    The boundedness of  $\{\la_n : n \ge 1\}$ in $C^{0,1}(U, \R^d)$ is a consequence of Bronshtein's theorem \ref{thm:Bronshtein}.
    It implies
    that $|\p_1 \la - \p_1 \la_n|$ is dominated on $U_0$ by an integrable function.
    By Fubini's theorem,
    \begin{align*}
        \int_{U_0} \|\p_1 \la(x) - \p_1 \la_n(x)\|_2^q\, dx &=  \int_{U_0'} A_n(x') \, d{x'}.
    \end{align*}
    By the dominated convergence theorem,
    we conclude that 
    \[
        \int_{U_0} \|\p_1 \la(x) - \p_1 \la_n(x)\|_2^q\, dx \to 0 \quad \text{ as } n \to \infty.
    \]
    In an analogous way, one sees that $\|\p_j \la - \p_j \la_n\|_{L^q(U_0, \R^d)} \to 0$ as $n \to \infty$, 
    for each $1\le j\le m$. 
    
    We may conclude that $\|\la - \la_n\|_{L^\infty(U_0,\R^d)} \to 0$ as $n\to \infty$ from the fact that this 
    is true component-wise (see Step 3 in \Cref{ssec:strategy}).

    For general $U_0$, we observe that there are finitely many open boxes as before that are relatively compact in $U$ and 
    cover $U_0$. This ends the proof.
\end{proof}

\subsection{Proof of \Cref{thm:solmap}}

It is clear that     
\Cref{thm:multhyp} implies \Cref{thm:solmap} 
because $C^d(U,\on{Hyp}(d))$ is first-countable. 

\subsection{Proof of \Cref{cor:solmap}}

\Cref{cor:solmap} is an immediate consequence of the following corollary of \Cref{thm:multhyp} 
and \Cref{ex:Br-continuity}.

\begin{corollary}
    Let $U \subseteq \R^m$ be open.
    Let $\aa_n \to \aa$ in $C^d(U,\on{Hyp}(d))$ as $n\to  \infty$.
    Then, for each relatively compact open set $U_0 \Subset U$ and each $0 < \al  < 1$,
    \begin{equation*} \label{eq:mainconcl3}
        \|\sol(\aa) -  \sol(\aa_n)\|_{C^{0,\al}(\ol U_0,\R^d)}  \to 0 \quad \text{ as } n\to \infty. 
    \end{equation*}
\end{corollary}

\begin{proof}
    Again we may assume that $U_0$ is a box (and hence has Lipschitz boundary).
    Then the assertion follows from \Cref{thm:multhyp} and Morrey's inequality,
    \[
    \|\sol(\aa)-  \sol(\aa_n)\|_{C^{0,\al}(\ol U_0,\R^d)} \le C\, \|\sol(\aa) -  \sol(\aa_n)\|_{W^{1,q}(U_0,\R^d)},
    \]
    where $\al = 1-m/q$, $q> m$, and $C=C(m,q,U_0)$.
\end{proof}

\subsection{Proof of \Cref{t:simple}}

The restriction $\ps:=\la^\uparrow|_{\on{Hyp}^\o(d)} : \on{Hyp}^\o(d) \to \R^d$ is real analytic, by \Cref{lem:splitting}.
Thus \Cref{t:simple} is a consequence of \Cref{prop:lefttr}, observing that $\sol^\o  = \ps_*$ and that $\|\ps\|_{C^{k+1}(\ol V_0,\R^d)}$ 
depends only on $d$, $k$, and $V_0$.

\section{Applications} \label{sec:app}

In this section, we give several applications of our results.
In \Cref{ssec:app1}, we clarify their relation to the continuity results for the solution map of general polynomials 
obtained in \cite{Parusinski:2024ab}. 
In \Cref{ssec:app2}, we deduce that locally the surface area of the graphs of the roots of hyperbolic polynomials 
is continuous and conclude local lower semicontinuity of the area of the zero sets of hyperbolic polynomials.  
In \Cref{ssec:app3}, we prove a theorem on approximation by hyperbolic polynomials with all roots simple.
Finally, we obtain continuity results for the eigenvalues of Hermitian matrices, in \Cref{ssec:Hermitian},
and for singular values, in \Cref{ssec:singval}.

\subsection{Relation to the results for general polynomials} \label{ssec:app1}
 
The case of general complex (not necessarily hyperbolic) polynomials is treated in \cite{Parusinski:2024ab}
which builds on the results of \cite{ParusinskiRainerAC,ParusinskiRainer15}.
The crucial difference is that in general there is no canonical choice of a continuous ordered $d$-tuple of the complex roots. 
Even worse, if the parameter space is at least two-dimensional, then a parameterization of the roots by continuous functions might not exist;
but there exist parameterizations by functions of bounded variation, see \cite{Parusinski:2020aa}.
Therefore the continuity results  in \cite{Parusinski:2024ab} are formulated in terms of the \emph{unordered} $d$-tuple of the roots.

Let us compare the results obtained in this paper with the ones of \cite{Parusinski:2024ab}. 
To this end, we investigate the metric space $\cA_d(\R)$ of unordered $d$-tuples of real numbers.
It is a simple instance of the space $\cA_d(\R^m)$ considered in \cite{Almgren00} and \cite{De-LellisSpadaro11}. 

For $x=(x_1,\ldots,x_d) \in \R^d$, let $[x] = [x_1,\ldots,x_d]$ be the corresponding unordered $d$-tuple, 
i.e., the equivalence class (or orbit) of $x$ with respect to the action 
of the symmetric group $\on{S}_d$ on $\R^d$ by permutation of the coordinates:
\[
    \si x := (x_{\si(1)},x_{\si(2)},\ldots,x_{\si(d)}),\quad \si \in \on{S}_d,~ x \in \R^d. 
\]

The set $\cA_d(\R) := \{[x] : x \in \R^d \}$ with the distance 
\[
    \mathbf d([x],[y]) := \min_{\si \in \on{S}_d}\frac{1}{\sqrt d} \, \|x - \si y\|_2
\]
is a complete metric space. If we identify the elements of $\cA_d(\R)$ with formal sums 
$\frac{1}{d} \sum_{i=1}^d \llbracket x_i \rrbracket$, where $\llbracket x_i \rrbracket$ denotes the Dirac mass of $x_i \in \R$,
then $\mathbf d$ is induced by the $L^2$ based Wasserstein metric on the space of 
probability measures on $\R$.

For $x \in \R^d$, let $x^\uparrow \in \R^d$ be the representative of the equivalence class $[x]$ with increasingly ordered coordinates. 
Clearly, $x^\uparrow$ only depends on $[x]$ and thus we have an injective map $(~)^\uparrow : \cA_d(\R) \to \R^d$.
It is a right-inverse of $[~] : \R^d \to \cA_d(\R)$. 

\begin{lemma} \label[l]{lem:unordered}
   We have 
   \[
       \mathbf d([x],[y]) = \frac{1}{\sqrt d} \, \|x^\uparrow - y^\uparrow \|_2, \quad x,y \in \R^d.
   \]
   In particular, $(~)^\uparrow : \cA_d(\R) \to \R^d$ and $[~] : \R^d \to \cA_d(\R)$ are Lipschitz maps.
\end{lemma}

\begin{proof}
   Evidently,
   \[
       \mathbf d([x],[y]) = \mathbf d([x^\uparrow],[y^\uparrow]) = \min_{\si \in \on{S}_d}\frac{1}{\sqrt d} \, \|x^\uparrow - \si y^\uparrow\|_2 \le \frac{1}{\sqrt d} \,\|x^\uparrow - y^\uparrow\|_2.
   \]
   Thus the assertion will follow from the claim that $\|x^\uparrow - y^\uparrow\|_2 \le \|x^\uparrow-y\|_2$, for all $x,y \in \R^d$.
   For $d=2$, the claim is equivalent to
   \[
       (x_1 - y_1)^2 + (x_2-y_2)^2 \le (x_1 - y_2)^2 + (x_2-y_1)^2
   \]
   whenever $x_1 \le x_2$ and $y_1 \le y_2$.
   By a simple computation, it is further equivalent to the true statement $(x_2-x_1)(y_2-y_1) \ge 0$.
   The general case follows from the fact that any permutation is a finite composite of transpositions.
\end{proof}

By \Cref{lem:unordered}, the map $(~)^\uparrow : \cA_d(\R) \to \R^d$ satisfies the conditions of an Almgren embedding 
(as defined in \cite{Parusinski:2024ab} following \cite{Almgren00} and \cite{De-LellisSpadaro11}). 
Thus \Cref{thm:mainhyp} can be interpreted as a special version of the general theorem \cite[Theorem 1.1]{Parusinski:2024ab}
with the important difference that $\sol(\aa_n) \to \sol(\aa)$ in $W^{1,q}_{\on{loc}}$ as $n \to \infty$, see
\eqref{eq:mainconcl}, holds for each $1 \le q <\infty$,
while in the general result the corresponding fact is valid only for $1 \le q < d/(d-1)$.

For the next theorem, which is a stronger version of \cite[Theorem 1.3]{Parusinski:2024ab} in the hyperbolic case,
we need to recall the notions of metric speed and $q$-energy.

\begin{definition}[Metric speed and $q$-energy]
Let $I \subseteq \R$ be an open interval and $\aa \in C^d(I,\Hyp(d))$.
Consider the Lipschitz curve $\La(x) := [\sol(\aa)(x)]$, for $x \in I$, in the metric space $\cA_d(\R)$.
Then (see \cite{Ambrosio:2008aa}) the limit 
\[
    |\dot \La|(x) := \lim_{h \to 0} \frac{\mathbf d(\La(x+h),\La(x))}{|h|}
\]
exists for almost every $x \in I$ and is called the \emph{metric speed} of $\La$ at $x$. 
The \emph{$q$-energy} of $\La$ on a subinterval $I_0 \subseteq I$ is defined by 
\[
\cE_{q,I_0}(\La) := \int_{I_0} \big(|\dot \La|(x) \big)^q \, dx.
\]
\end{definition}

\begin{theorem} \label[t]{thm:unordered}
    Let $I \subseteq \R$ be an open interval.
    Let $\aa_n \to \aa$ in $C^d (I, \on{Hyp}(d))$ as $n\to \infty$.
    Set $\La := [\sol(\aa)]$ and $\La_n := [\sol(\aa_n)]$.
    Then, for each relatively compact open interval $I_0 \Subset I$, 
    \begin{align} \label{eq:mse1}
        \|\mathbf{d}(\La,\La_n) \|_{L^\infty(I_0)} &\to 0 \quad \text{ as } n \to \infty,
        \\ \label{eq:mse2}
        \big\| |\dot \La| - |\dot \La_n| \big\|_{L^q(I_0)} &\to 0 \quad \text{ as } n \to \infty,
        \\ \label{eq:mse3}
        \big| \cE_{q,I_0}(\La)- \cE_{q,I_0}(\La_n) \big| &\to 0 \quad \text{ as } n \to \infty,
    \end{align}
    for each $1 \le q <\infty$.
\end{theorem}

\begin{proof}
    First, \eqref{eq:mse1} is a consequence of \eqref{eq:C0conv} and \Cref{lem:unordered}.
    By \cite[Lemma 11.1]{Parusinski:2024ab}, 
    \[
        |\dot \La|(x) = \frac{1}{\sqrt d}\, \|\sol(\aa)'(x)\|_2 
        \quad \text{ and } \quad
        |\dot \La_n|(x) = \frac{1}{\sqrt d}\, \|\sol(\aa_n)'(x)\|_2 
    \]
    for almost every $x \in I$. 
    Thus, \eqref{eq:mse2} and \eqref{eq:mse3} follow from \Cref{cor:se}.
\end{proof}

\subsection{Continuity of the area of the solution map} \label{ssec:app2}

Let us first expand \Cref{cor:se}.

\begin{corollary} \label[c]{cor:area1}
   Let $U \subseteq \R^m$ be open.
    Let $\aa_n \to \aa$ in $C^d(U,\on{Hyp}(d))$ as $n\to \infty$. 
    Let $R \in \R[X_1,\ldots,X_{dm}]$ be any real polynomial in the $d \cdot m$ variables $X_1,\ldots,X_{dm}$.
    Set $\la = (\la_1,\ldots,\la_d) := \sol(\aa)$ and 
    $\la_n = (\la_{n,1},\ldots,\la_{n,d}) := \sol(\aa_n)$, for $n \ge 1$.
    Then, for each relatively compact open subset $U_0 \Subset U$ and each $1 \le q<\infty$,
    \begin{equation} \label{eq:R1}
        \Big\|R\Big((\p_i \la_j)_{\substack{1\le i\le m\\ 1\le j \le d}}\Big) - R\Big((\p_i\la_{n,j})_{\substack{1\le i\le m\\ 1\le j \le d}}\Big)\Big\|_{L^q(U_0)} \to 0 \quad \text{ as } n \to \infty,
    \end{equation}
    and consequently,
    \begin{equation} \label{eq:R2}
      \Big\| R\Big((\p_i\la_{n,j})_{\substack{1\le i\le m\\ 1\le j \le d}}\Big)\Big\|_{L^q(U_0)}   \to  \Big\|R\Big((\p_i \la_j)_{\substack{1\le i\le m\\ 1\le j \le d}}\Big) \Big\|_{L^q(U_0)}\quad \text{ as } n \to \infty.
    \end{equation}
\end{corollary}

\begin{proof}
    Clearly, \eqref{eq:R2} is a consequence of \eqref{eq:R1}. 
    
    Let us prove \eqref{eq:R1}.
    It is enough to show the assertion for monomials $R$.
    Let us proceed by induction on the degree $\ell$ of the monomial $R$.
    For $\ell =1$, the assertion follows from \Cref{thm:multhyp} in view of 
    \begin{align*}
        \|\p_i \la_j - \p_i \la_{n,j}\|_{L^q(U_0)} \le         
        \|\p_i \la - \p_i \la_{n}\|_{L^q(U_0,\R^d)}.
    \end{align*}
    If $\ell \ge 2$, then, by H\"older's inequality, 
    \begin{align*}
       \MoveEqLeft \| \p_{i_1} \la_{j_1} \cdots \p_{i_\ell} \la_{j_\ell}-\p_{i_1} \la_{n,j_1} \cdots \p_{i_\ell} \la_{n,j_\ell}\|_{L^q(U_0)}
       \\
        &\le
        \| \p_{i_1} \la_{j_1} \cdots \p_{i_\ell} \la_{j_\ell} -\p_{i_1} \la_{j_1} \cdots \p_{i_{\ell-1}} \la_{j_{\ell-1}} \cdot\p_{i_\ell} \la_{n,j_\ell}\|_{L^q(U_0)}
        \\
        &\quad +
        \| \p_{i_1} \la_{j_1} \cdots \p_{i_{\ell-1}} \la_{j_{\ell-1}} \cdot\p_{i_\ell} \la_{n,j_\ell}   -\p_{i_1} \la_{n,j_1} \cdots \p_{i_\ell} \la_{n,j_\ell}\|_{L^q(U_0)}
       \\
        &\le
        \|\p_{i_1} \la_{j_1} \cdots \p_{i_{\ell-1}} \la_{j_{\ell-1}}\|_{L^\infty(U_0)} \|  \p_{i_\ell} \la_{j_\ell} -\p_{i_\ell} \la_{n,j_\ell}\|_{L^q(U_0)}
        \\
        &\quad +
        \| \p_{i_1} \la_{j_1} \cdots \p_{i_{\ell-1}} \la_{j_{\ell-1}}    -\p_{i_1} \la_{n,j_1} \cdots \p_{i_{\ell-1}} \la_{n,j_{\ell-1}}\|_{L^q(U_0)} \|\p_{i_\ell} \la_{n,j_\ell}\|_{L^\infty(U_0)}
    \end{align*}
    which tends to zero as $n \to 0$, by the induction hypothesis, because
    \[
        \|\p_{i_1} \la_{j_1} \cdots \p_{i_{\ell-1}} \la_{j_{\ell-1}}\|_{L^\infty(U_0)} \le C \quad \text{ and } \quad   \|\p_{i_\ell} \la_{n,j_\ell}\|_{L^\infty(U_0)} \le C
    \]
    for a constant $C>0$ independent of $n$ and $i_k,j_k$,
    by Bronshtein's theorem (see \Cref{thm:Bronshtein}).
\end{proof}

Let $f : U \to \R^d$ 
be a Lipschitz map, 
where $U \subseteq \R^m$ is open. 
We recall that the \emph{Jacobian} $|J f|$ of $f$ is the square root of the sum of the squares of the determinants of the $k \times k$ minors with $k = \min\{m,d\}$ of the Jacobian matrix
\[
    (\p_i f_j)_{\substack{1\le i\le m\\ 1\le j \le d}},
\]
which exists almost everywhere, by Rademacher's theorem.

\begin{corollary} \label[c]{cor:area2}
   Let $U \subseteq \R^m$ be open.
    Let $\aa_n \to \aa$ in $C^d(U,\on{Hyp}(d))$ as $n\to \infty$.
    Then, for each relatively compact open subset $U_0 \Subset U$ and each $1 \le q<\infty$,
    \[
        \big\| |J (\sol(\aa))| - |J (\sol(\aa_n))| \big\|_{L^q(U_0)}  \to 0 \quad \text{ as } n \to \infty,
    \]
    and consequently,
    \[
        \big\|   |J (\sol(\aa_n))| \big\|_{L^q(U_0)}  \to \big\| |J (\sol(\aa))| \big\|_{L^q(U_0)} \quad \text{ as } n \to \infty.
    \]
\end{corollary}

\begin{proof}
    Let $M_1,\ldots,M_p$ and $M_{n,1},\ldots,M_{n,p}$ denote the determinants of all the $k \times k$ minors with $k = \min\{m,d\}$ 
    of the Jacobian matrices of $\sol(\aa)$ and $\sol(\aa_n)$, respectively.
    Fix $1 \le q<\infty$.
    Then, by H\"older's inequality, 
\begin{equation*}
    \big\| |J (\sol(\aa))| - |J (\sol(\aa_n))| \big\|_{L^q(U_0)} 
    \le |U_0|^{1/(2q)} \, \big\| |J (\sol(\aa))| - |J (\sol(\aa_n))| \big\|_{L^{2q}(U_0)}
\end{equation*}
and
    \begin{align*}
        \MoveEqLeft \big\| |J (\sol(\aa))| - |J (\sol(\aa_n))| \big\|_{L^{2q}(U_0)}^{2q}
        = 
        \big\| \big(\sum_i M_i^2\big)^{1/2} -\big(\sum_i M_{n,i}^2\big)^{1/2}  \big\|_{L^{2q}(U_0)}^{2q}
        \\
        &\le 
        \big\| \big| \sum_i M_i^2 -\sum_i M_{n,i}^2\big|^{1/2}  \big\|_{L^{2q}(U_0)}^{2q}   
        = 
        \big\| \sum_i M_i^2 -\sum_i M_{n,i}^2  \big\|_{L^{q}(U_0)}^{q}.
    \end{align*}
    Now it suffices to apply \Cref{cor:area1}.
\end{proof}

Next we will combine \Cref{cor:area2} with the area and the coarea formula (see e.g.\ \cite{EvansGariepy92})
which we recall for the convenience of the reader.

Let $f : \R^m \to \R^d$ be Lipschitz and let $E \subseteq \R^m$ be Lebesgue measurable. 
The \emph{area formula} states that, 
if $m \le d$, then 
\[
    \int_E |J f| \, dx = \int_{\R^d} \cH^0(E \cap f^{-1}(y)) \, d\cH^m(y).
\]
The \emph{coarea formula} posits that,
if $m \ge d$, then 
\[
    \int_E |J f| \, dx = \int_{\R^d} \cH^{m-d}(E \cap f^{-1}(y)) \, dy.
\]
Recall that $\cH^k$ denotes the $k$-dimensional Hausdorff measure, 
in particular, $\cH^0$ is the counting measure.

\begin{corollary}
   Let $U \subseteq \R^m$ be open.
    Let $\aa_n \to \aa$ in $C^d(U,\on{Hyp}(d))$ as $n\to \infty$.
    Set $\la := \sol(\aa)$ and $\la_n := \sol(\aa_n)$, for $n \ge 1$.
    \begin{enumerate}
        \item If $m\le d$, then  for each relatively compact open subset $U_0 \Subset U$,
            \[
                \int_{\R^d} \cH^0(U_0 \cap \la_n^{-1}(y)) \, d\cH^{m}(y) \to \int_{\R^d} \cH^0(U_0 \cap \la^{-1}(y)) \, d\cH^{m}(y)
            \]
            as $n \to \infty$.
        \item If $m>d$, then for each relatively compact open subset $U_0 \Subset U$,
    \[
        \int_{\R^d} \cH^{m-d}(U_0 \cap \la_n^{-1}(y)) \, dy \to \int_{\R^d} \cH^{m-d}(U_0 \cap \la^{-1}(y)) \, dy
    \]
            as $n \to \infty$.
    \end{enumerate}
\end{corollary}

\begin{proof}
    This is an immediate consequence of \Cref{cor:area2} (for $q=1$) and the area and coarea formula.
\end{proof}

We can also conclude that the surface area of the graphs of all the single roots $\sol(\aa_n)_j =\la^\uparrow_j \o \aa_n$, for $1 \le j \le d$,
is locally convergent as $n \to \infty$.

\begin{corollary} \label[c]{cor:area4}
    Let $U \subseteq \R^m$ be open.
    Let $\aa_n \to \aa$ in $C^d(U,\on{Hyp}(d))$ as $n\to \infty$.
    For each $1 \le j \le d$ and for each relatively compact open subset $U_0 \Subset U$, 
    the surface area of the graph of $\la_{n,j} := \sol(\aa_n)_j$ converges to the surface area 
    of the graph of $\la_j:=\sol(\aa)_j$ as $n \to \infty$:
    if $\ol \la_{n,j}(x) := (x,\la_{n,j}(x))$ and $\ol \la_{j}(x) := (x,\la_j(x))$ denote 
    the corresponding graph mappings,
    then 
    \begin{align*}
        \cH^m(\ol \la_{n,j}(U_0)) \to  \cH^m(\ol \la_{j}(U_0)) \quad \text{ as }  n\to \infty.
    \end{align*}
\end{corollary}

\begin{proof}
    We have
    \[
        |J\ol \la_{j}| = \Big(1 + \sum_{i=1}^m (\p_i \la_{j})^2\Big)^{1/2}
        \quad \text{ and } \quad
        |J\ol \la_{n,j}| = \Big(1 + \sum_{i=1}^m (\p_i \la_{n,j})^2\Big)^{1/2}.
    \]
    As in the proof of \Cref{cor:area2}, we have
    \begin{align*}
      \MoveEqLeft  \Big\| \Big(1 + \sum_{i=1}^m (\p_i \la_{j})^2\Big)^{1/2}- \Big(1 + \sum_{i=1}^m (\p_i \la_{n,j})^2\Big)^{1/2}\Big\|_{L^2(U_0)}^2
      \\
      &\le  \Big\| \Big|\sum_{i=1}^m (\p_i \la_{j})^2-   \sum_{i=1}^m (\p_i \la_{n,j})^2\Big|^{1/2}\Big\|_{L^2(U_0)}^2
      \\
      &= \Big\| \sum_{i=1}^m (\p_i \la_{j})^2-   \sum_{i=1}^m (\p_i \la_{n,j})^2\Big\|_{L^1(U_0)}.
    \end{align*}
    So the assertion follows from \Cref{cor:area1} and the area formula.
\end{proof}

It follows that the area of the zero sets of $C^d$ families of hyperbolic polynomials of degree $d$ 
locally has a lower semicontinuity property; 
for the reader's convenience, we restate \Cref{cor:area5intro}:

\begin{corollary} \label[c]{cor:area5}
    Let $U \subseteq \R^m$ be open.
    Let $\aa_n \to \aa$ in $C^d(U,\on{Hyp}(d))$ as $n\to \infty$.
    For any relatively compact open $U_0 \Subset U$, 
            consider the zero sets
            \begin{align*}
                Z &= \{ (x,y) \in U_0 \times \R : P_{\aa(x)}(y)=0\} \quad \text{ and }
                \\
                Z_n &= \{ (x,y) \in U_0 \times \R : P_{\aa_n(x)}(y)=0\}, \quad n\ge 1.
            \end{align*}
    Then 
    \begin{align*}
        \liminf_{n \to \infty}  \cH^m(Z_n) \ge  \cH^m(Z).
    \end{align*}
\end{corollary}

\begin{proof}
    Set $\la := \sol(\aa)$.
    For $i=2,\ldots,d$, let $E_{i} := \{x \in U_0 : \la_{i-1}(x) = \la_{i}(x)\}$. 
    Then, using the notation of \Cref{cor:area4},
    \begin{equation} \label{eq:HmZ}
        \cH^m(Z) = \cH^m(\ol \la_1(U_0)) + \sum_{i=2}^d \cH^m(\ol \la_i(U_0\setminus E_i)). 
    \end{equation}
    Analogously, setting $\la_n := \sol(\aa_n)$ and $E_{n,i} := \{x \in U_0 : \la_{n,i-1}(x) = \la_{n,i}(x)\}$, we have 
    \begin{equation} \label{eq:HmZn}
        \cH^m(Z_n) = \cH^m(\ol \la_{n,1}(U_0)) + \sum_{i=2}^d \cH^m(\ol \la_{n,i}(U_0\setminus E_{n,i})). 
    \end{equation}

    By the continuity of $\la^\uparrow : \on{Hyp}(d) \to \R^d$, for each $i =2,\ldots,d$ and each $x \in U_0$,
    \begin{equation} \label{eq:limsup1}
        \limsup_{n\to \infty} \mathbf 1_{E_{n,i}}(x) \le \mathbf 1_{E_i}(x).
    \end{equation}
   
    By \Cref{thm:ptwhyp} (applied coordinate by coordinate), we have that $|J \ol \la_{n,i}| \to |J \ol \la_i|$ as $n \to \infty$ 
    almost everywhere in $U_0$. By Bronshtein's theorem (\Cref{thm:Bronshtein}),  
    there is a constant $B>0$ such that 
    $\big\| |J \ol \la_{n,i}| \big\|_{L^\infty(U_0)} \le B$ for all $n \ge 1$.
    Thus,
    by the area formula and the reverse Fatou lemma, 
    \begin{align*}
        \limsup_{n\to \infty} \cH^m(\ol \la_{n,i}(E_{n,i})) 
        &=  \limsup_{n\to \infty} \int_{U_0} \mathbf 1_{E_{n,i}} |J \ol \la_{n,i}|\, dx
        \\
        &\le   \int_{U_0} \limsup_{n\to \infty}\big(\mathbf 1_{E_{n,i}} |J \ol \la_{n,i}| \big)\, dx
        \\
        &\le   \int_{U_0} \limsup_{n\to \infty} \mathbf 1_{E_{n,i}} \cdot \limsup_{n\to \infty} |J \ol \la_{n,i}| \, dx
        \\
        &\le   \int_{E_i}   |J \ol \la_{i}| \, dx = \cH^m(\ol \la_i(E_i)),
    \end{align*}
    where we used \eqref{eq:limsup1} in the last inequality.

    Together with \eqref{eq:HmZ}, \eqref{eq:HmZn}, and \Cref{cor:area4}, this gives 
    \begin{align*}
        \liminf_{n \to \infty}  \cH^m(Z_n) 
        &= \liminf_{n \to \infty} \Big( \sum_{i=1}^d \cH^m(\ol \la_{n,i}(U_0)) - \sum_{i=2}^d \cH^m(\ol \la_{n,i}(E_{n,i}))  \Big)
        \\
        &=  \sum_{i=1}^d \liminf_{n \to \infty}\cH^m(\ol \la_{n,i}(U_0)) - \sum_{i=2}^d \limsup_{n\to \infty} \cH^m(\ol \la_{n,i}(E_{n,i})) 
        \\
        &\ge \sum_{i=1}^d \cH^m(\ol \la_{i}(U_0)) - \sum_{i=2}^d \cH^m(\ol \la_i(E_i))
        = \cH^m(Z)
    \end{align*}
    which ends the proof.
\end{proof}

\subsection{Approximation by hyperbolic polynomials with all roots distinct} \label{ssec:app3}

We recall a lemma of Wakabayshi \cite{Wakabayashi86} which extends an observation of Nuij \cite{Nuij68}. 

\begin{lemma}[{\cite[Lemma 2.2]{Wakabayashi86}}] \label[l]{lem:Wakabayashi}
    Let $P_\aa \in \on{Hyp}(d)$ and set 
    \begin{equation} \label{eq:diffop}
        P_{\aa,s}(Z) :=  (1 + s \tfrac{\p}{\p Z})^{d-1} P_\aa(Z), \quad s \in \R.
    \end{equation}
    Then $P_{\aa,s} \in \on{Hyp}(d)$ for all $s \in \R$ and there are positive constants $c_i=c_i(d)$, $i=1,2$,
    such that, if $\la^\uparrow_1(\aa,s) \le \cdots \le \la^\uparrow_d(\aa,s)$ denote the increasingly ordered roots of $P_{\aa,s}$, then 
    \begin{equation} \label{eq:Waka1}
        \la^\uparrow_j (\aa,s) - \la^{\uparrow}_{j-1}(\aa,s) \ge c_1 |s|, \quad \text{ for } s \in \R \text{ and } 2 \le j \le d,   
    \end{equation}
    and 
    \begin{equation}
        0 < \pm (\la^\uparrow_j(\aa) - \la^\uparrow_j(\aa,s)) \le c_2 |s|, \quad \text{ for } \pm s >0 \text{ and } 1 \le j \le d.
    \end{equation}
\end{lemma}

In conjunction with our findings, \Cref{lem:Wakabayashi} leads to the following approximation result.

\begin{corollary} \label[c]{cor:approx}
    Let $U \subseteq \R^m$ be open and 
    $\aa \in C^d(U, \on{Hyp}(d))$. There exists a sequence $(\aa_n)_{n\ge 1} \subseteq C^d(U, \on{Hyp}(d))$ with the following properties: 
    \begin{enumerate}
        \item $\aa_n \to \aa$ in $C^d(U,\on{Hyp}(d))$ as $n \to \infty$;
        \item $\sol(\aa_n)_1(x) < \sol(\aa_n)_2(x) < \cdots < \sol(\aa_n)_d(x)$ for all $x \in U$ and all $n \ge 1$;
        \item $\sol(\aa_n) \in C^d(U,\R^d)$ for all $n \ge 1$;   
        \item $\sol(\aa_n) \to \sol(\aa)$ in $C^{0,1}_{q}(U,\R^d)$, for all $1 \le q<\infty$, as $n \to \infty$; 
        \item for any relatively compact open $U_0 \Subset U$, 
            consider the zero sets
            \begin{align*}
                Z &= \{ (x,y) \in U_0 \times \R : P_{\aa(x)}(y)=0\} \quad \text{ and }
                \\
                Z_n &= \{ (x,y) \in U_0 \times \R : P_{\aa_n(x)}(y)=0\}, \quad n\ge 1.
            \end{align*}
            Then $\lim_{n \to \infty} \cH^m(Z_n)$ exists and
            \[
                \lim_{n \to \infty} \cH^m(Z_n) \ge \cH^m(Z).
            \]
    \end{enumerate}
\end{corollary}

\begin{proof}
    Let $(s_n)_{n\ge 1}$ be any positive sequence of reals that tends to $0$. 
    Consider the polynomial $P_{\aa(x),s_n}$ (defined in \eqref{eq:diffop}), where $x \in U$, and let $\aa_n(x)$ be its coefficient vector.
    Then, by \Cref{lem:Wakabayashi},
    $\aa_n \in C^d(U,\on{Hyp}(d))$, for $n\ge 1$.
    We will show that the sequence $(\aa_n)_{n\ge 1}$ has the desired properties.

    (1) This is clear by the definition \eqref{eq:diffop} and since $s_n \to 0$ as $n \to \infty$.

    (2) follows from \eqref{eq:Waka1} and the fact that $s_n>0$ for all $n \ge 1$.

    (3) For fixed $x \in U$, $\frac{\p}{\p Z} P_{\aa_n(x)}(Z)$ does not vanish at any root of $P_{\aa_n(x)}$, by (2).
    So, by the implicit function theorem, the roots of $P_{\aa_n(x)}$ are of class $C^d$ in a neighborhood of $x$.
    This implies (3).

    (4) is a consequence of (1) and \Cref{thm:solmap}.

    (5) Using the notation of \Cref{cor:area4}, for each $n\ge 1$, the set 
    $Z_n$ is the union of the graphs $\ol \la_{n,j}(U_0)$ of the single roots $\la_{n,j}|_{U_0} = \sol(\aa_n)_j|_{U_0}$, for $1 \le j \le d$, 
    and these graphs are pairwise disjoint, by (2).
    Thus, by (1) and \Cref{cor:area4}, 
    \begin{align*}
        \lim_{n\to \infty} \cH^m(Z_n) 
        = \lim_{n \to \infty}\sum_{j=1}^d \cH^m(\ol \la_{n,j}(U_0)) = \sum_{j=1}^d \cH^m(\ol \la_j(U_0)) \ge \cH^m(Z). 
    \end{align*}
    For the inequality at the end, note that the union $Z = \bigcup_{j=1}^d \ol \la_j(U_0)$ is not necessarily disjoint; 
    see also \Cref{cor:area5}.
\end{proof}

\subsection{Perturbation theory for Hermitian matrices} \label{ssec:Hermitian}

Let $\on{Herm}(d)$ denote the real vector space of complex Hermitian $d \times d$ matrices.
With $A \in \on{Herm}(d)$ we associate its increasingly ordered eigenvalues 
$\la_1^\uparrow(A) \le \la_2^\uparrow(A) \le \cdots \le \la_d^\uparrow(A)$
and thus obtain a continuous map
\begin{equation} \label{eq:Herm}
    \la^\uparrow = (\la_1^\uparrow, \ldots, \la_d^\uparrow) : \on{Herm}(d) \to \R^d.
\end{equation}

\begin{proposition}[Weyl's perturbation theorem \cite{Weyl12}; see e.g.\ {\cite[III.2.6]{Bhatia97}}] \label[p]{p:Weyl}
    Let $A, B \in \on{Herm}(d)$. Then 
    \begin{equation} \label{eq:Weyl}
       \|\la^\uparrow(A)-\la^\uparrow(B)\|_\infty \le \|A-B\|,
   \end{equation}
   where $\|A-B\|$ denotes the operator norm of $A-B$.
\end{proposition}

In \cite{Bhatia97}, the result is stated for eigenvalue vectors with decreasing eigenvalues, but reversing the order evidently 
leaves the left-hand side of \eqref{eq:Weyl} unchanged.

As a consequence of \Cref{p:Weyl},
the map \eqref{eq:Herm} induces a bounded map
\begin{equation} \label{eq:E}
   \eig:= (\la^\uparrow)_* : C^{0,1}(\ol I, \on{Herm}(d)) \to C^{0,1}(\ol I, \R^d), \quad A \mapsto \la^\uparrow \o A,
\end{equation}
which takes Lipschitz curves of Hermitian matrices to Lipschitz curves of their increasingly ordered eigenvalues. The Lipschitz constants 
satisfy
\begin{equation} \label{eq:eigLip}
    |\eig(A)|_{C^{0,1}(\ol I,\R^d)} \le |A|_{C^{0,1}(\ol I,\on{Herm}(d))},
\end{equation}
if $\on{Herm}(d)$ is endowed with the operator norm and $\R^d$ with the maximum norm.
This remains true if we replace the interval $I$ by a bounded open set  $U \subseteq \R^m$.

The following corollary includes \Cref{t:eigmap}.

\begin{corollary} \label[c]{c:eigmap}
    Let $U \subseteq \R^m$ be open.  Then:
    \begin{enumerate}
        \item 
    The map    
    \begin{equation*} 
        \eig : C^{d}(U,\on{Herm}(d)) \to C^{0,1}_q(U,\R^d), \quad A \mapsto \la^\uparrow \o A,
    \end{equation*}
    is continuous, for all $1\le q<\infty$.
    \item
    The map    
    \begin{equation*} 
        \eig : C^{d}(U,\on{Herm}(d)) \to C^{0,\al}(U,\R^d), \quad A \mapsto \la^\uparrow \o A,
    \end{equation*}
    is continuous, for all $0< \al<1$.
    \item
      If $A_n \to A$ in $C^{d}(U,\on{Herm}(d))$ as $n \to \infty$, then over each relatively compact open subset $U_0 \Subset U$ 
      the surface area of the graph of $\eig(A_n)_j$ converges to the  surface area of the graph of $\eig(A)_j$ as $n \to \infty$, for each $1 \le j \le d$.
    \end{enumerate}
\end{corollary}

\begin{proof}
    (1) We have the commuting diagram
    \[
        \xymatrix{
            C^{d}(U,\on{Herm}(d)) \ar[dr]_{\cP} \ar[rr]^{\eig} && C^{0,1}_q(U,\R^d)
            \\
                                                 &  C^{d}(U,\on{Hyp}(d)) \ar[ru]_{\sol} &
        }
    \]
    where $\cP$ sends $A$ to its characteristic polynomial $P_A$. The coefficients of $P_A$ are given by 
    polynomials in the entries of $A$. Thus, \Cref{prop:lefttr} implies that the map $\cP$ is continuous.
    Consequently, $\eig = \sol \o \cP$ is continuous by \Cref{thm:solmap}.

    (2) This follows similarly from the continuity of $\cP$ and \Cref{cor:solmap}.

    (3) Use \Cref{cor:area4} and the continuity of $\cP$.
\end{proof}

The following example shows that $\eig$ is not continuous with respect to the 
$C^{0,1}$ topology on the target space.

\begin{example} \label[e]{ex:A}
    The sequence $(A_n)_n$ of curves of symmetric $2 \times 2$ matrices 
    \[
        A_n(x) = \begin{pmatrix} \frac{1}n & x \\ x & - \frac{1}{n}  \end{pmatrix}, \quad x \in \R,
    \]
    converges to 
    \[
        A(x) = \begin{pmatrix} 0 & x \\ x & 0  \end{pmatrix}, \quad x \in \R,
    \]
    uniformly on every compact interval in all derivatives. 
    We have 
    \[
        \eig(A_n)(x) = \Big(-\sqrt{x^2 + \frac{1}{n^2}}, \sqrt{x^2 + \frac{1}{n^2}}\Big)
    \]
    and 
    \[
        \eig(A)(x) = (-|x|,|x|).
    \]
    Hence, \Cref{ex:Br-continuity} shows that the Lipschitz constant of $\eig(A)-\eig(A_n)$ on each bounded open interval containing $0$ 
    is bounded below by $2-\sqrt 2$
    which shows that $\eig(A_n) \not\to \eig(A)$ in the $C^{0,1}$ topology.
\end{example}

Given that the map $\eig$ is defined and bounded on $C^{0,1}(U,\on{Herm}(d))$ (see \eqref{eq:E} and \eqref{eq:eigLip}),
it is natural to ask whether in \Cref{c:eigmap} one can replace $C^d$ by $C^1$, see \Cref{q:AC1}.

If $d=2$, this is indeed the case as evidenced in the following proposition.

\begin{proposition} \label[p]{p:Ad=2}
    Let $U \subseteq \R^m$ be open.  Then:
    \begin{enumerate}
        \item 
    The map    
    \begin{equation*} 
        \eig : C^{1}(U,\on{Herm}(2)) \to C^{0,1}_q(U,\R^2), \quad A \mapsto \la^\uparrow \o A,
    \end{equation*}
    is continuous, for all $1\le q<\infty$.
    \item
    The map    
    \begin{equation*} 
        \eig : C^{1}(U,\on{Herm}(2)) \to C^{0,\al}(U,\R^2), \quad A \mapsto \la^\uparrow \o A,
    \end{equation*}
    is continuous, for all $0< \al<1$.
    \item
      If $A_n \to A$ in $C^{1}(U,\on{Herm}(2))$ as $n \to \infty$, then over each relatively compact open subset $U_0 \Subset U$ 
      the surface area of the graph of $\eig(A_n)_j$ converges to the  surface area of the graph of $\eig(A)_j$ as $n \to \infty$, for $j=1,2$.
    \end{enumerate}
\end{proposition}

\begin{proof}
    It suffices to prove (1) in the case $m=1$. Then the multiparameter version of (1) as well as (2) and (3) follow by the arguments given in detail for hyperbolic polynomials, 
    if one uses \Cref{p:Weyl} instead of Bronshtein's theorem \ref{thm:Bronshtein}.

    Let us show (1) for $m=1$.
    We may assume that the trace of $A$ vanishes, by replacing $A$ by $A - \frac{1}{2}\on{tr}(A)\I$. Thus, we have 
    \[
        A = 
        \begin{pmatrix}
            a & b+i c \\ b-i c & -a
        \end{pmatrix},
    \]
    where $a,b,c \in \R$, and 
    \[
        \la^\uparrow(A) = \big(- \sqrt{a^2 + b^2 +c^2}, \sqrt{a^2 + b^2 +c^2}\big).
    \]
    
    Let us assume that $a,b,c \in C^1(I,\R)$, where $I \subseteq \R$ is an open interval. Then
    \[
        I \ni x \mapsto \sqrt{a(x)^2 + b(x)^2 +c(x)^2} =: \mu(x)
    \]
    is locally Lipschitz, differentiable almost everywhere, and 
    \begin{equation} \label{eq:mu}
        |\mu'(x_0)| \le \sup_{x \in \ol I_1}\|A'(x)\|_2 = \sqrt 2 \,\sup_{x \in \ol I_1} \sqrt{a'(x)^2 + b'(x)^2 + c'(x)^2},
    \end{equation}
    for each relatively compact open subinterval $I_1 \Subset I$ and each $x_0 \in I_1$ where $\mu'(x_0)$ exists, by \eqref{eq:eigLip}.

    Suppose that 
    \[
        A_n = 
        \begin{pmatrix}
            a_n & b_n+i c_n \\ b_n-i c_n & -a_n
        \end{pmatrix}
        \longrightarrow 
        A = 
        \begin{pmatrix}
            a & b+i c \\ b-i c & -a
        \end{pmatrix} \quad \text{ in } C^1(I,\on{Herm}(2))
    \]
    as $n \to \infty$. 
    It suffices to prove that 
    \begin{equation} \label{eq:muW}
        \|\mu - \mu_n\|_{W^{1,q}(I_1)} \to 0 \quad \text{ as } n \to \infty,
    \end{equation}
    for each relatively compact open subinterval $I_1 \Subset I$ and all $1 \le q < \infty$, where 
    \[
        \mu_n := \sqrt{a_n^2 + b_n^2 +c_n^2}.
    \]

    By \eqref{eq:Weyl}, we have 
    \[
        \|\mu - \mu_n\|_{L^\infty(I_1)} \to 0 \quad \text{ as } n \to \infty,
    \]
    for each relatively compact open subinterval $I_1 \Subset I$.

    We claim that, for almost every $x \in I$, 
    \[
        \mu_n'(x) \to \mu'(x) \quad \text{ as } n \to \infty.
    \]

    This is clear on the set $\Om := \{x \in I : a(x)^2 + b(x)^2 +c(x)^2 \ne 0\}$: for each $x_0 \in \Om$, the derivative $\mu'(x_0)$ exists, and, by assumption, 
    $a(x_0)^2 + b(x_0)^2 +c(x_0)^2 \ne 0$ if $n$ is large enough so that also $\mu_n'(x_0)$ exists and $\mu_n'(x_0) \to \mu'(x_0)$. 

    Now consider $Z := \{x \in I : a(x)^2 + b(x)^2 +c(x)^2 = 0\}$ and the set $\on{acc}(Z)$ of accumulation points of $Z$. 
    Note that $a'$, $b'$, and $c'$ vanish on $\on{acc}(Z)$.
    Fix $x_0 \in \on{acc}(Z)$ and $\ep>0$. By continuity, there exists $\de>0$ such that $\ol I(x_0,\de) \Subset I$ and 
    \[
        \sup_{x \in \ol I(x_0,\de)} \sqrt{a'(x)^2 + b'(x)^2 + c'(x)^2} \le \frac{\ep}2.
    \]
    As $A_n \to A$ in $C^1(I,\on{Herm(2)})$, there is $n_0 \ge 1$ such that, for all $n \ge n_0$, 
    \[
        \sup_{x \in \ol I(x_0,\de)} \sqrt{a_n'(x)^2 + b_n'(x)^2 + c_n'(x)^2} \le \ep.
    \]
    If $\mu_n'(x_0)$ exists, then we conclude, by \eqref{eq:mu}, that 
    \[
        |\mu_n'(x_0)| \le \sqrt 2\, \ep, \quad n \ge n_0.  
    \] 
    This implies the claim, since the set of accumulation points of $Z$, where all $\mu_n$ and $\mu$ are differentiable, has full measure in $Z$
    and $\mu'$ vanishes on this set.

    Now the dominated convergence theorem implies that 
    \[
        \|\mu'- \mu_n'\|_{L^q(I_1)} \to 0 \quad \text{ as } n \to \infty,
    \]
    for each relatively compact open subinterval $I_1 \Subset I$ and all $1 \le q < \infty$, 
    completing the proof of \eqref{eq:muW}.
\end{proof}

\begin{remark}\label[r]{rem:eigenvalues}
    \Cref{c:eigmap} has an evident analogue for skew-Hermitian matrices which simply follows from the fact 
    that a $d \times d$ matrix $A$ is skew-Hermitian if and only if $i A$ is Hermitian. The eigenvalues of $A$ 
    and $iA$ just differ by multiplication by $i$.

    On the other hand, there is no consistent continuous choice of the eigenvalues of unitary $d \times d$ matrices. 
    Consider, for example, 
    the curve of unitary matrices 
    \[
       A(x) = \begin{pmatrix} 0 & e^{2 \pi i x} \\ 1 & 0 \end{pmatrix}, \quad x \in \R, 
    \]
    with the eigenvalues $\la_{\pm}(x) = \pm e^{\pi i x}$. Even though $\la_{\pm} : \R \to (\mathbb S^1)^2$ is continuous, 
    there is no continuous choice of the eigenvalues $\mathbb S^1 \to (\mathbb S^1)^2$ of 
    the curve of unitary matrices induced by $A$ on $\mathbb S^1 = \R/\Z$ (because 
    $\la_{\pm}(0) = \pm 1 \ne \mp 1 = \la_{\pm}(1)$). 

    In this case, and more generally for normal matrices, the general continuity results of \cite{Parusinski:2024ab}
    apply. For the perturbation theory of normal matrices, see \cite{RainerN}, \cite{Parusinski:2020ac}, and the 
    survey \cite{Parusinski:2023ab}.
\end{remark}

\subsection{Singular values} \label{ssec:singval}

Let us consider the vector space $M_{D,d}(\C)$ of complex $D \times d$ matrices, where $d \le D$. 
The singular values of $A \in M_{D,d}(\C)$ are the nonnegative square roots of the eigenvalues 
of the Hermitian matrix $A^*A$, usually ordered decreasingly
\[
    \si_1(A) \ge \si_2(A) \ge \cdots \ge \si_d(A) \ge 0.
\]
This defines a map $\si=(\si_1,\ldots,\si_d) : M_{D,d}(\C) \to \R^d$.

Let us consider the real vector space $M_{D,d}(\C) \times \R$ and the homogeneous polynomial 
of degree $2d$,
\[
    f(A,r) := \det(r^2 \, \mathbb I - A^*A ), \quad (A,r) \in M_{D,d}(\C) \times \R.
\]
Then $f$ is \emph{G\r{a}rding hyperbolic} with respect to the direction $(0,1)\in M_{D,d}(\C) \times \R$ 
(see \cite[Sec.\ 6]{BGLS01})
which means by definition that all roots of the univariate polynomial
\[
    P_{A,r}(Z) :=  f((A,r) - Z\, (0,1)) = \det((r-Z)^2\, \mathbb I - A^*A) 
\]
are real. Indeed, the roots of $P_{A,r}$ (in decreasing order) are
\[
    r+\si_1(A), r+\si_2(A), \ldots, r+\si_d(A), r-\si_d(A), \ldots,r-\si_1(A). 
\]
Hence, by \Cref{thm:Bronshtein}, $\si$ induces a bounded map
\begin{equation*} 
    \si_* : C^{2d-1,1}(U,M_{D,d}(\C)) \to C^{0,1}(U,\R^d), \quad A \mapsto \si \o A,
\end{equation*}
where $U \subset \R^m$ is open. In general, this map is not continuous, which follows from \Cref{ex:A},
but we have the following result.

\begin{corollary} \label[c]{c:singval}
    Let $U \subseteq \R^m$ be open.  Then:
    \begin{enumerate}
        \item 
    The map    
    \begin{equation*} 
        \si_* : C^{2d}(U,M_{D,d}(\C)) \to C^{0,1}_q(U,\R^d), \quad A \mapsto \si \o A,
    \end{equation*}
    is continuous, for all $1\le q<\infty$.
    \item
    The map    
    \begin{equation*} 
    \si_* : C^{2d}(U,M_{D,d}(\C)) \to C^{0,\al}(U,\R^d), \quad A \mapsto \si \o A,
    \end{equation*}
    is continuous, for all $0< \al<1$.
    \item
      If $A_n \to A$ in $C^{2d}(U,M_{D,d}(\C))$ as $n \to \infty$, 
      then over each relatively compact open subset $U_0 \Subset U$ 
      the surface area of the graph of $\si_j(A_n)$ converges to the  surface area of the graph of $\si_j(A)$ 
      as $n \to \infty$, for each $1 \le j \le d$.
    \end{enumerate}
\end{corollary}

\begin{proof}
    Similarly as in the proof of \Cref{c:eigmap}, we have, for each $r \in \R$, a 
    continuous map $C^{2d}(U,M_{D,d}(\C)) \to C^{2d}(U,\on{Hyp}(d))$, $A \mapsto P_{A,r}$, 
    which can be used to reduce the statements of the corollary to the corresponding one for 
    hyperbolic polynomials.
\end{proof}

Observing that the Hermitian matrix 
\[
   \mathbf A:= \begin{pmatrix}
        0 & \tilde A
        \\
        \tilde A^* & 0
    \end{pmatrix},
\]
where $\tilde A$ is the $D \times D$ matrix resulting from $A$ by adding $D-d$ columns consisting of zeros, has the eigenvalues 
\[
    \si_1(A), \ldots,\si_d(A), 0,\ldots,0, -\si_d(A),\ldots,-\si_1(A),
\]
we conclude from \eqref{eq:Weyl} that, for $A,B \in M_{D,d}(\C)$ and $1 \le i \le d$, 
\begin{align*}
    |\si_i(A) - \si_i(B)| &\le \|\mathbf A - \mathbf B\| \le \|\mathbf A - \mathbf B\|_2 
    = |\on{tr}\big( (\mathbf A - \mathbf B)^*(\mathbf A - \mathbf B)\big)|^{1/2}
    \\
                          &= |2\, \on{tr}\big( (\tilde A - \tilde B)^*(\tilde  A - \tilde B)\big)|^{1/2} 
                         = \sqrt 2 \, \|A-B\|_2.
\end{align*}
Consequently, the map
\[
    \si_* : C^{0,1}(U,M_{D,d}(\C)) \to C^{0,1}(U,\R^d)
\]
is well-defined and bounded. 
So, in analogy to \Cref{q:AC1}, it is thus natural to ask whether in \Cref{c:singval} one can replace the assumption $C^{2d}$ by $C^1$.

\section{Restricted multiplicity} \label{sec:rm}

In this section we prove a refinement of \Cref{thm:mainhyp} which accounts for the case that the 
maximal multiplicity of the roots is smaller than the degree. 

First we recall the following version of Bronshtein's theorem. 

\begin{theorem}[{\cite[Theorem 2.1]{ParusinskiRainerHyp}}] \label[t]{t:Brm}
    Let $I \subseteq \R$ be an open interval 
    and $\aa \in C^{p-1,1}(I,\on{Hyp}(d))$, where 
    $p$ is the maximal multiplicity of the roots of $P_{\aa(x)}$, for $x \in I$.
    Then any continuous root $\la \in C^0(I)$ of $P_{\aa}$ is locally Lipschitz. 

    If $p=d$, then we have the bound \eqref{eq:Lipconst}.
  
    Assume $p<d$ and suppose that $P_{\tilde \aa}$ is in Tschirnhausen form. 
    Let $\la_1^{\uparrow}(x) \le \cdots \le \la_d^{\uparrow}(x)$ be the increasingly ordered 
    roots of $P_{\tilde \aa(x)}$, for $x \in I$, and consider
    \begin{equation} \label{eq:alpha}
        \al(x) := \frac{|\la_d^{\uparrow}(x)-\la_1^\uparrow(x)|}{\min_{1 \le i \le d-p}|\la_{i+p}^\uparrow(x)-\la_i^\uparrow(x)|} 
        \quad \text{ and } \quad \al_I := \sup_{x \in I} \al(x).
    \end{equation}
    Then each continuous root $\la$ of $P_{\tilde a}$ satisfies, for any pair of relatively compact 
    open intervals $I_0 \Subset I_1 \Subset I$,
    \begin{align*}
        |\la|_{C^{0,1}(\ol I_0)} &\le C(d)\, \al_{I_1}^{\frac{d-p}{p}}\, 
        \max\Big\{ \de^{-1} \|\tilde a_2\|^{1/2}_{L^\infty(I_1)}, |\tilde a_2'|_{C^{0,1}(\ol I_1)}^{1/2}, 
        \\
                                 & \hspace{3.2cm} \max_{i \le p} \Big( |\tilde a_i^{(p-1)}|_{C^{0,1}(\ol I_1)} \|\tilde a_2\|^{\frac{p-i}{2}}_{L^\infty(I_1)}\Big)^{1/p},
                                 \\
                                 & \hspace{3.2cm} \max_{i > p} \Big( |\tilde a_i^{(p-1)}|_{C^{0,1}(\ol I_1)} \big(\min_{x \in \ol I_0} |\tilde a_2(x)|\big)^{\frac{p-i}{2}}\Big)^{1/p} \Big\},
    \end{align*}
where $\de := \on{dist}(I_0, \R \setminus I_1)$.
\end{theorem}

The next theorem generalizes \Cref{thm:mainhyp}.

\begin{theorem} \label[t]{t:rm}
    Let $I \subseteq \R$ be an open interval.
    Let $\aa_n \to \aa$ in $C^p (I, \on{Hyp}(d))$ as $n \to \infty$,
    where $p$ is the maximal multiplicity of the roots of $P_{\aa(x)}$, for $x \in I$.
    If $p<d$ assume that, for each relatively compact open $I_1 \Subset I$, 
    \begin{equation*}
    \al_{I_1} < \infty.
    \end{equation*}
    Then $\{\sol(\aa_n) : n \ge 1\}$ is a bounded set in $C^{0,1}(I, \R^d)$ 
    and, for each relatively compact open interval $I_0 \Subset I$ and each $1 \le q < \infty$,
    \begin{equation*} 
        \|\sol(\aa) -  \sol(\aa_n)\|_{W^{1,q}(I_0,\R^d)}  \to 0 \quad \text{ as } n\to \infty. 
    \end{equation*}
\end{theorem}

\begin{proof}
    In view of \Cref{thm:mainhyp}, we may assume that $p< d$. Furthermore, we may assume that all polynomials are 
    in Tschirnhausen form.
    
    We first observe that, for each relatively compact open $I_1 \Subset I$, 
    \begin{equation} \label{eq:uc}
        \|\sol(\tilde \aa) - \sol(\tilde \aa_n)\|_{L^\infty(I_1,\R^d)} \to 0 \quad \text{ as } n \to \infty,
    \end{equation}
    as a consequence of \cite[Corollary 6.5]{Parusinski:2024ab} and \Cref{lem:unordered}. 
    (For this it is actually enough that $\tilde \aa_n \to \tilde \aa$ in $C^0(\ol I_1,\on{Hyp}_T(d))$ 
    as $n \to \infty$.)

    Fix relatively compact open subintervals $I_0 \Subset I_1 \Subset I$.
    Then there exists $n_0 \ge 1$ such that for all $n \ge n_0$ 
    the maximal multiplicity of the roots of $P_{\tilde \aa_n}$ on $I_1$ is at most $p$.
    (If not this is violated on a sequence $x_{n_k}$ in $I_1$, leading to a contradiction at an 
    accumulation point of this sequence in $\ol I_1$, since $\la^\uparrow : \on{Hyp}_T(d) \to \R^d$ is continuous.)

    Consequently, the functions $\al_n : I_1 \to \R$
    associated to $P_{\tilde \aa_n}$ as in \eqref{eq:alpha} are well-defined, for all $n \ge n_0$. 
    By the assumption $\al_{I_1} < \infty$ and \eqref{eq:uc}, 
    $\al_n \to \al$ uniformly on $I_1$ as $n \to \infty$ and thus the sequence 
    $\al_{n,I_1} := \sup_{x \in I_1} \al_n(x)$ is bounded.

    Hence, by \Cref{t:Brm}, the derivative of $\sol(\aa_n)$ exists almost everywhere in $I_0$ and 
    is uniformly bounded on $I_0$ by a constant independent of $n$.

    By \Cref{lem:splitting} and \Cref{prop:lefttr}, 
    we can split $P_{\tilde \aa}$ and $P_{\tilde \aa_n}$, for large $n$,  
    locally in factors of degrees at most $p$ in a simultaneous way. This allows us to apply \Cref{thm:ptwhyp} 
    in the case $d=p$ and conclude that, for almost every $x \in I_0$, 
    \[
        \sol(\aa_n)'(x) \to \sol(\aa)'(x) \quad \text{ as } n \to \infty.
    \]
    Now it suffice to invoke the dominated convergence theorem to finish the proof.
\end{proof}

\subsection*{Acknowledgement}

We are grateful to Antonio Lerario for posing the question about the continuity of the solution map.

This research was funded 
in part by the Austrian Science Fund (FWF) DOI 10.55776/P32905 and DOI 10.55776/PAT1381823.
For open access purposes, the authors have applied a CC BY public copyright license to any author-accepted manuscript version arising from this submission.


\def\cprime{$'$}
\providecommand{\bysame}{\leavevmode\hbox to3em{\hrulefill}\thinspace}
\providecommand{\MR}{\relax\ifhmode\unskip\space\fi MR }
\providecommand{\MRhref}[2]{%
  \href{http://www.ams.org/mathscinet-getitem?mr=#1}{#2}
}
\providecommand{\href}[2]{#2}

\end{document}